\documentclass[12pt,leqno]{amsart}
\usepackage{amsmath, amssymb, amscd, amsfonts, xypic,   stmaryrd, turnstile, mathrsfs, eucal,color}

\definecolor{dullmagenta}{rgb}{0.4,0,0.4}

\definecolor{darkblue}{rgb}{0,0,0.4}

\usepackage[colorlinks=true, pdfstartview=FitV, linkcolor=red, citecolor=blue, urlcolor=darkblue]
{hyperref}

\newtheorem{theorem}{Theorem}[section]
\newtheorem{lemma}[theorem]{Lemma}
\newtheorem{proposition}[theorem]{Proposition}
\newtheorem{corollary}[theorem]{Corollary}

\theoremstyle{definition}
\newtheorem{definition}[theorem]{Definition}
\newtheorem{example}[theorem]{Example}
\newtheorem{remark}[theorem]{Remark}

\begin{document}

\title[Boolean FIP]{Boolean FIP ring extensions}

\author[G. Picavet and M. Picavet]{Gabriel Picavet and Martine Picavet-L'Hermitte}
\address{Math\'ematiques \\
8 Rue du Forez, 63670 - Le Cendre\\
 France}
\email{picavet.mathu (at) orange.fr}

\begin{abstract} We characterize extensions of commutative rings $R \subseteq S$ whose sets of subextensions $[R,S]$ are finite ({\it i.e.} $R\subseteq S$ has the FIP property) and are Boolean lattices, that we call Boolean FIP extensions.  Some characterizations involve ``factorial" properties of the poset $[R,S]$. A non trivial result is that each subextension of a Boolean FIP extension is simple (i.e. $R \subseteq S$ is a simple pair).  
\end{abstract}


\keywords  {FIP, FCP extension, minimal extension, integral extension, support of a module, lattice, Boolean lattice, algebraic field extension, Galois extension}

\maketitle

\section{Introduction and Notation}

We consider the category of commutative and unital rings, whose    epimorphisms will be involved. If  $R\subseteq S$ is a (ring) extension, we denote  by $[R,S]$ the set of all $R$-subalgebras of $S$ and
  set $]R,S[: =[R,S]\setminus \{R,S\}$ (with a similar definition for $[R,S[$ or $]R,S]$). 
 
  A {\it lattice} is a poset $L$ such that every pair $a,b\in L$ has a supremum and an infimum. 
    For an extension $R\subseteq S$, the poset $([R,S],\subseteq)$ is a {\it complete} lattice where the supremum of any non void subset  is the compositum which we call {\it product} from now on and denote by $\Pi$ when necessary, and the infimum of any non void subset is the intersection. We are aiming  to study some lattice properties of the poset $([R,S], \subseteq)$, mainly the Boolean property.  As a general rule, an extension $R\subseteq S$ is said to have some property of lattices if $[R,S]$ has this property.

The extension $R\subseteq S$ is said to have FIP (for the ``finitely many intermediate algebras property") or an FIP  extension if $[R,S]$ is finite. A {\it chain} of $R$-subalgebras of $S$ is a set of elements of $[R,S]$ that are pairwise comparable with respect to inclusion.  When $[R,S]$ is a chain, the extension $R\subseteq S$ is  called a $\lambda$-{\it extension} by some authors.  We will say that $R\subseteq S$ is chained. We  also say that the extension $R\subseteq S$ has FCP  (or is an FCP extension) if each chain in $[R,S]$ is finite. Clearly,  each extension that satisfies FIP must also satisfy FCP. 
Dobbs and the authors characterized FCP and FIP extensions \cite{DPP2}. 

Let $R\subseteq S$ be an FCP extension, then $[R,S]$ is a complete Noetherian Artinian lattice, with $R$ as the least element and $S$ as the largest element. We use lattice definitions and properties described in \cite{NO}.    

Our main tool are the minimal (ring) extensions, a concept that was introduced by Ferrand-Olivier \cite{FO}. Recall that an extension $R\subset S$ is called {\it minimal} if $[R, S]=\{R,S\}$. 
  An extension $R\subseteq S$ is called a {\it simple} extension if $S=R[t]$ for some $t\in S$ and a {\it simple pair} if $R\subseteq T$ is a simple extension for each $T\in [R,S]$. A minimal extension is simple. The key connection between the above ideas is that if $R\subseteq S$ has FCP, then any maximal (necessarily finite) chain of $R$-subalgebras of $S$, $R=R_0\subset R_1\subset\cdots\subset R_{n-1}\subset R_n=S$, with {\it length} $n <\infty$, results from juxtaposing $n$ minimal extensions $R_i\subset R_{i+1},\ 0\leq i\leq n-1$. 
An FCP extension is finitely generated, and  (module) finite if integral.
For any extension $R\subseteq S$, the {\it length} $\ell[R,S]$ of $[R,S]$ is the supremum of the lengths of chains of $R$-subalgebras of $S$. Notice  that if $R\subseteq S$ has FCP, then there {\it does} exist some maximal chain of $R$-subalgebras of $S$ with length $\ell[R,S]$ \cite[Theorem 4.11]
{DPP3}.

\subsection{ A summary of the main results} Any undefined  material
 is explained at the end of the section or in the next sections. 

Section 2 is  devoted  to some general properties of  lattices $[R,S]$, mainly in the context of  FCP and FIP extensions.  
 Since Boolean extensions are distributive, we have evidently to work on distributive extensions, which is done in this section.
We  discuss the decomposition of elements of $[R,S]$ into irreducible elements. 
    When $[R,S]$ has finitely many atoms and each element of $[R,S]$ is a product of atoms, then Proposition \ref{6.111} shows that $R\subseteq S$ has FIP and is almost-Pr\"ufer,  and Theorem \ref{5.8} shows that $R\subseteq S$ is a simple  pair.  
 
For extensions $R\subseteq S$ of integral domains, Ayache considered  {\it Boolean lattices} 
  (also called Boolean algebras) 
 $([R,S],\cap,\cdot)$, that are distributive lattices such that each $T\in[R,S]$ has a (unique) complement  \cite{A}, \cite{A1} and \cite{A2}. In particular,  \cite[Problem 45]{A1} asked under which conditions $[R,S]$ is a finite Boolean lattice. This question is completely answered in Sections 3 and  4, where we get in 
  Theorem \ref{4.18} 
  a characterization of  Boolean extensions.   
In particular, Theorem \ref{4.0} shows that an FCP Boolean extension $R\subseteq S$ has FIP, each element of $[R,S]$ has a unique representation by a finite product of atoms, and $R\subset S$ a simple pair.
  Section 3 is devoted to the study of arbitrary FIP extensions.   The canonical decomposition of a ring extension  is crucial. It consists of the tower  $R \subseteq {}_S^+R\subseteq {}_S^tR \subseteq \overline R \subseteq S$, where $ {}_S^+R$ (resp. $ {}_S^tR $) is the {\it seminormalization} (resp.  $t$-{\it closure}) of $R$ in $S$ (see  Section~3).
 This decomposition  
allows us to only  consider special  extensions: subintegral, seminormal infra-integral, t-closed  and integrally closed. The t-closed  case is  reduced  to the context of field extensions and is the subject of Section 4. In particular, for a field extension $k\subset L$ with separable  closure $T$ and radicial closure $U$ such that $U,T\not\in\{k,L\}$, Theorem \ref{4.199} shows that $k\subset L$ is Boolean if and only if   $k\subset U$ and $T\subset L$ are minimal, $[k,L]=[k,T]\cup[U,L],\ k\subset T$ and $k\subset U$ are linearly disjoint and  $[k,T]$ is a Boolean lattice. Boolean separable field extensions needs a special study. A striking result is Theorem \ref{4.37}:  A Galois finite extension (hence FIP) $k\subset L$ is Boolean if and only if $k\subset L$ is a cyclic extension whose dimension is square free.
   \subsection{Some conventions and notation} 
  A {\it local} ring is here what is called elsewhere a quasi-local ring. As usual, Spec$(R)$ and Max$(R)$ are the set of prime and maximal ideals of a ring $R$. The support of an $R$-module $E$ is $\mathrm{Supp}_R(E):=\{P\in\mathrm{Spec }(R)\mid E_P\neq 0\}$, and $\mathrm{MSupp}_R(E):=\mathrm{Supp}_R(E)\cap\mathrm{Max}(R)$. When $R 
 \subseteq S$ is an extension, we will set   $\mathrm{Supp}_R(T/R):= \mathrm{Supp}(T/R)$  and $\mathrm{Supp}_R(S/T):= \mathrm{Supp}(S/T)$ for each $T\in [R,S]$, unless otherwise specified.

 If $R\subseteq S$ is a ring extension and $P\in\mathrm{Spec}(R)$, then $S_P$ is both the localization $S_{R\setminus P}$ as a ring and the localization at $P$ of the $R$-module $S$. 
  We denote by $\kappa_R(P)$ the residual field $R_P/PR_P$ at $P$.
   An extension $R\subset S$ is called {\it locally minimal} if $R_P\subset S_P$ is minimal for each $P\in\mathrm{Supp}(S/R)$  or equivalently for each $P\in\mathrm{MSupp}(S/R)$.

We denote by $(R:S)$ the conductor of $R\subseteq S$. The integral closure of $R$ in $S$ is denoted by $\overline R^S$ (or by $\overline R$ if no confusion can occur).

  Recall  that an  extension $R\subseteq S$ is {\it Pr\"ufer} (or  a normal pair) if $R\subseteq T$ is a flat epimorphism for each $T\in[R,S]$. The {\it Pr\"ufer hull} of an extension $R\subseteq S$ is the greatest {\it Pr\"ufer} subextension $\widetilde R$ of $[R,S]$ \cite{Pic 3}. An extension $R\subseteq S$ is called {\it almost-Pr\"ufer} if $\widetilde R\subseteq S$ is integral, or equivalently, when $R\subseteq S$ is FIP, if $S=\widetilde R\overline R$ \cite[Theorem 4.6]{Pic 5}.

 A poset $(X,\leq)$ is called a {\it tree} if $x_1,x_2\leq x_3$ in $X$ implies that $x_1$ and $x_2$ are comparable (with respect to $\leq$). We also say that $(X,\leq)$ is treed. A subset $Y$ of $X$ is called an {\it antichain} if no two distinct elements of $Y$ are comparable. 
  
 Finally,  $|X|$  is the cardinality of a set $X$,   $\subset$ denotes proper inclusion and, for a positive integer $n$, we set $\mathbb{N}_n:=\{1,\ldots,n\}$.  
   The characteristic of an integral domain $k$ is denoted by $\mathrm{c}(k)$.
 For $a,b,c$ in a ring $R$,  if $c$ divides $a-b$, we  write $a\equiv b\ (c)$.
  
\section {Lattices properties of the poset [R,S]}

\subsection {Some definitions on the lattice [R,S]} 

In the context of a lattice $[R,S]$, some  definitions and properties of lattices have the following formulations.

An element $T\in[R,S]$ is called $\cap$-{\it irreducible} (resp$.$; $\Pi$-{\it irreducible}) (see \cite{NO}) if $T=T_1\cap T_2$ (resp$.$; $T=T_1T_2$) implies $T=T _1$ or $T=T_2$. 

 An element $T$ of $[R,S]$ is an {\it atom} (resp.; {\it co-atom}) if and only if $R\subset T$ (resp.; $T\subset S$) is a minimal extension. 
  Therefore, an atom (resp.; co-atom) is $\Pi$-irreducible (resp.; $\cap$-irreducible). We denote by $\mathcal{A}$ the set of atoms of $[R,S]$ and by $\mathcal{CA}$ the set of co-atoms of $[R,S]$. 

Now $R\subset S$ is called: 

(a) {\it atomic} (resp.; {\it atomistic}) if each $T\in ]R,S]$ contains some atom (resp.; is the product of atoms (contained in $T$)) \cite[page 80]{R}. 

(b) {\it co-atomic} (resp.; {\it co-atomistic}) if each $T\in [R,S[$ is contained in some co-atom (resp.; is the intersection  of    co-atoms (containing $T$)).

(c) {\it distributive} if intersection and product are each distributive with respect to the other. Actually, each distributivity implies the other \cite[Exercise 5, page 33]{NO}.

 (d) {\it factorial} (resp.; {\it co-factorial}) if each element of $[R,S]$ has a unique irredundant representation by a finite product of atoms (resp.; a unique irredundant representation by a finite intersection  of co-atoms.) 

An FCP extension is both atomic and co-atomic.

  We introduce  a  definition reminiscent of arithmetic rings  \cite{Pic 4}.

\begin{definition}\label{4.2} A ring extension $R\subseteq S$ is called 
{\it arithmetic} if $[R_P, S_P]$ is a chain for each $P\in\mathrm{Spec}(R)$.
\end{definition}

 \begin{example} The extension $R\subset S$ is arithmetic in the following cases (\cite[Example 5.13]{Pic 4} for (2), (3) and (4)):
\begin{enumerate}
\item $R\subset S$ is locally minimal.

\item $R\subset S$ has FCP and is integrally closed.

\item $R\subset S$ is FIP subintegral and  $|R/M|=\infty$ for each $M\in\mathrm{MSupp}(S/R)$.

\item $R\subset S$ is FIP t-closed integral such that $R_M/MR_M\subset S_M/MR_M$ is radicial for each $M\in\mathrm{MSupp}(S/R)$.

\item We also have examples in \cite[Theorem 6.1]{Pic 6} of arithmetic extensions of length 2 $R\subset S$ when $|[R,S]|=3$.
\end{enumerate}
\end{example}

We proved in \cite[Proposition 5.18]{Pic 4}: 

\begin{proposition}\label{5.4} An arithmetic extension is distributive. 
\end{proposition}  

 The following proposition will make easier many proofs.
 
 \begin{proposition}\label{1.014} Let $R\subseteq S$ be a ring  extension. The following statements are equivalent: 
\begin{enumerate}
\item $R\subseteq  S$ is distributive ;

\item $R_M \subseteq S_M$ is distributive for each $M\in\mathrm{MSupp}(S/R)$:

\item $R_P \subseteq S_P$ is distributive for each $P\in\mathrm{Supp}(S/R)$:

\item $R/I\subseteq S/I$ is  distributive  for each  ideal $I$ shared by $R$ and $S$.

\item  $R/I\subseteq S/I$ is  distributive  for some  ideal $I$ shared by $R$ and $S$.
\end{enumerate}
\end{proposition}
\begin{proof} We have obviously (1) $\Rightarrow$ (3) $\Rightarrow$ (2) and (1) $\Leftrightarrow$ (4) 
$\Rightarrow$ (5) $\Rightarrow$ (1). Conversely, assume that $R_M \subseteq S_M$ is distributive for each $M\in\mathrm{MSupp}(S/R)$. Then, $R_M \subseteq S_M$ is   distributive for each $M\in\mathrm{Max}(R)$. It follows that the distributivity property holds in $[R,S]$ since it holds in any $[R_M,S_M]$. 
\end{proof}
 
 \begin{proposition} \label{1.0} \cite[Theorem 1, p. 172]{G} In a distributive lattice of finite length, all maximal chains between two comparable elements have the same length (the Jordan-H\"older chain condition or condition (JH)).
\end{proposition} 

 In \cite{DS2}, Dobbs and Shapiro defined a field extension $k\subset L$ 
  of finite length
 to be {\it catenarian} if all maximal chains  of fields between $k$ and $L$ have the same length, and get examples of catenarian field extensions. It follows that a distributive field extension is catenarian.
  
    \begin{proposition} \label{cover} \cite[Remarks, page 9 and Theorem 1.7]{Cal} A distributive extension $R\subset S$ satisfies the {\it upper covering condition} 
     (UCC),
which means that for each $T,U\in[R,S]$ such that $T\cap U\subset T$ is minimal, then $U\subset TU$ is minimal.
\end{proposition}  

 \subsection{Some distributive extensions are simple} 
 We are going to show that   some special subextensions of an FCP distributive extension are simple. Before, next lemma is needed.
 
\begin{lemma}\label{1.03}  If $R\subset S$ is a  distributive extension, then $R[x,y]=R[x+y]$ whenever  $y\in S\setminus R$, $x\in S\setminus R[y]$  and $R\subset R[x]$ is minimal.  
\end{lemma}   

\begin{proof} Consider the diagram
$$\begin{matrix}
   {}  &        {}      & R[x,y] &       {}       & {}     \\
   {}  & \nearrow & {}        & \nwarrow & {}     \\
R[x] &       {}       & {}       &      {}         & R[y] \\
  {}   & \nwarrow & {}       & \nearrow  & {}      \\
  {}   &      {}        & R       & {}             & {} 
\end{matrix}$$
Since $R[x]R[y]=R[x,y]$ and 
 $R\subset R[x]$ is minimal, we get that 
 $R[y]\subset R[x,y]$ is minimal by  
  UCC. 
 There is another  diagram 
 
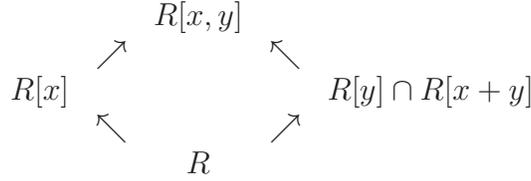
\centerline{$\begin{matrix}
   {}  &        {}      & R[x,y] &       {}       & {}                       \\
   {}  & \nearrow & {}        & \nwarrow & {}                       \\
R[x] &       {}       & {}       &      {}         & R[y]\cap R[x+y] \\
  {}   & \nwarrow &{}        & \nearrow  & {}                        \\
  {}   &      {}        & R       & {}             & {} 
\end{matrix}$}
\noindent where $R[x]\cap (R[y]\cap R[x+y])=R$ and $R\subset R[x]$ is minimal. Using again 
 UCC,
 we get that $R[y]\cap R[x+y]\subset R[x](R[y]\cap R[x+y])$ is minimal. But $R[x](R[y]\cap R[x+y])=R[x]R[y]\cap R[x]R[x+y]=R[x,y]$ by distributivity, so that $R[y]\cap R[x+y]\subset R[x,y]$ is minimal. From $R[y]\cap R[x+y]\subseteq R[y]\subset R[x,y]$, we deduce   $R[y]= R[y]\cap R[x+y]$, which implies $R[y]\subset R[x+y]\subseteq R[x,y]$ because $x\not\in R[y]$, and then $R[x+y]= R[x,y]$.   
\end{proof} 

\begin{proposition}\label{1.4}  Let $R\subset S$ be a distributive extension. Let $T\in[R,S]$ be a product of finitely many atoms. Then, $R\subset T$ is simple. More precisely, if $T=\prod_{i=1}^nR[x_i]$, where the $R\subset R[x_i]$ are minimal distinct extensions, then, $T=R[\sum _{i=1}^nx_i]$.  
\end{proposition} 

\begin{proof} We prove the two statements by induction on $n$. There is nothing to prove when $n=1$. Assume that the induction hypothesis holds for $n-1$ and set $T':=\prod_{i=1}^{n-1}R[x_i]\neq R$, so that $T'=R[x]$ with $x:=\sum _{i=1}^{n-1}x_i$. Then, $T=R[x]R[x_n]$, with $x\not\in R$ and $x_n\in S\setminus R[x]$. Deny. Then $R[x_n]\subseteq \prod_{i=1}^{n-1}R[x_i]$ would imply $R[x_n]\subseteq R[x_i]$ for some $i\in\{1,\ldots,n-1\}$, a contradiction \cite[Theorem 4.30]{R}. Now, use Lemma \ref{1.03} to get the result.
\end{proof}

 \subsection {The 
 lattice
 [R,S] for an  FCP or FIP extension } 

 \begin{proposition}\label{5.5} \cite[Proposition 1.4.4]{NO} If $R\subseteq S$ has FCP, then any $T\in [R,S]$ is a finite intersection (resp$.$; product) of $\cap$-irreducible (resp$.$; $\Pi$-irreducible) elements of $[R,S]$.
\end{proposition}

\begin{lemma}\label{5.7} 
If $R \subseteq S$ is a  
 distributive extension  and $T \in [R,S]$ has an irredundant representation $T=T_1\cdots T_m$ 
 by $\Pi$-irreducible elements of $[R,S]$, 
(resp.; $T=U_1\cap \cdots\cap T_r$ by $\cap$-irreducible elements of $[R,S]$) (for example by atoms (resp.; co-atoms)), 
  the representation is unique. 
 
If in addition $R \subseteq S$ has FCP,  $[R,S]$ has exactly $n\ \Pi$-irreducible (resp.; $\cap$-irreducible) elements if and only if $\ell[R,S]=n$. 
 If these  conditions hold,  $R \subseteq S$ has FIP.
 \end{lemma}
 
  \begin{proof}   \cite[Theorem 4.30]{R},   \cite[Theorem 147]{Do} and \cite[Lemma 4, page 486]{BM} combine to yield the result. The FIP property is then  an easy consequence. 
  \end{proof}
 
 \begin{theorem}\label{5.8} An extension  $R\subset S$ is   atomistic, distributive  and  FIP if and only if $R\subset S$ is factorial (respectively, is co-factorial). 
   If these  conditions hold, then $R\subset S$ is a simple pair. 
\end{theorem}
\begin{proof} If $R\subset S$ is an atomistic distributive FIP extension, then $[R,S]$ has finitely many atoms, and then is factorial by Lemma \ref{5.7}.

Conversely, assume that $R\subset S$ is factorial, whence atomistic. Let $T,U,V\in[R,S]$. Obviously, $(T\cap U)(T\cap V)\subseteq T\cap UV$. Since $R\subset S$ is factorial, we can write $T=\prod_{\alpha\in I}A_{\alpha},\ U=\prod_{\beta\in J}A_{\beta}$ and $V=\prod_{\gamma\in K}A_{\gamma}$, for finite subsets $\{A_{\alpha}\mid \alpha\in I\},\ \{A_{\beta}\mid \beta\in J\}$ and $\{A_{\gamma}\mid \gamma\in K\}$ of $\mathcal{A}$. Then, $UV=\prod_{\beta\in J\cup K}A_{\beta}$, so that $T\cap UV=(\prod_{\alpha\in I}A_{\alpha})\cap(\prod_{\beta\in J\cup K}A_{\beta})$. Write $T\cap UV=:\prod_{\delta\in L}A_{\delta}=(\prod_{\alpha\in I}A_{\alpha})\cap(\prod_{\beta\in J\cup K}A_{\beta})$. Then, for any $\delta\in L$, we have $A_{\delta}\subseteq T$ and $A_{\delta}\subseteq\prod_{\beta\in J\cup K}A_{\beta}$, so that there exist some $\alpha\in I$ such that $A_{\delta}=A_{\alpha}$ and some  $\beta\in J\cup K$ such that $A_{\delta}=A_{\beta}$. If $\beta\in J$, then $A_{\delta}\subseteq U$, so that $A_{\delta}\subseteq T\cap U\subseteq(T\cap U)(T\cap V)$. If  $\beta\in K$, then $A_{\delta}\subseteq V$, so that $A_{\delta}\subseteq T\cap V\subseteq(T\cap U)(T\cap V)$. In both cases, $A_{\delta}\subseteq (T\cap U)(T\cap V)$, which yields $T\cap UV\subseteq (T\cap U)(T\cap V)$, and then $T\cap UV= (T\cap U)(T\cap V)$. Therefore, $R\subset S$ is  distributive. Moreover, $R\subset S$ has FIP since  $\mathcal{A}$ is finite  ($S$ is the product of all elements of $\mathcal{A}$).

Now, let $R\subset S$ be factorial, whence distributive. Set $n:=|\mathcal{A}|$, and for $A_{\alpha}\in\mathcal{A}$, set $B_{\alpha}:=\prod_{\beta\neq \alpha}A_{\beta}$. Obviously, $\mathcal{CA}=\{B_{\alpha}\mid \alpha\in\mathbb N_n\}$. Let $T\in [R,S]$, with $T=\prod_{\alpha\in I}A_{\alpha}$. For $J=\mathbb N_n\setminus I$, an easy calculation shows  that $T=\cap_{\beta\in J}B_{\beta}$ in a unique way. Hence,  $R\subset S$ is co-factorial.

To end, assume that $R\subset S$ is co-factorial. We get that $R\subset S$ is factorial and distributive, mimicking the previous proof. It is enough to exchange product and intersection, and atoms and co-atoms. In fact, we use the fact that $R\subset S$ is co-atomistic.
 If these  conditions hold, then $R\subset S$ is a simple pair by Proposition  \ref{1.4}. 
\end{proof}

The following notions and results are  deeply involved in the sequel. 
\begin{definition}\label{crucial 1}\cite[Theorem 4.5]{CDL} An extension $R\subset S$ is called  {\it $M$-crucial} if there is some  unique $M\in\mathrm{Max}(R)$, called the {\bf crucial (maximal) ideal}  $\mathcal{C}(R,S)$ of $R\subset S$, such that  $R_P=S_P$ for each $P\in\mathrm{Spec}(R)\setminus\{M\}$. 
\end{definition}

\begin{theorem}\label{crucial}\cite[Th\'eor\`eme 2.2]{FO} A  minimal extension  is crucial  and is either integral (finite) or a flat epimorphism.
\end{theorem}

\begin{lemma}\label{1.9} \cite[Corollary 3.2]{DPP2} If there exists a maximal chain $R=R_0 \subset\cdots\subset  R_i \subset\cdots \subset R_n=S$ of extensions, where $R_i\subset R_{i+1}$ is minimal, then  $\mathrm{Supp}(S/R)=\{ \mathcal C (R_i, R_{i+1})\cap R\mid i=0,\ldots,n-1\}$.
\end{lemma}

\begin{lemma} \label{1.10} \cite[Lemma 1.8]{Pic 6} Let $R\subset S$ be an FCP extension. and $M\in\mathrm{MSupp}(S/R)$. There is some $T\in[R,S]$ such that $R\subset T$ is minimal and $\mathcal{C}(R,T)=M$. 
\end{lemma}

If $R\subseteq S$ has FCP, set $\mathcal T:=\{T\in[R,S] | \mathrm{MSupp}_R(T/R)|= 1\}$ and  
${\mathcal T}_ M:=\{T\in\mathcal T\mid\mathrm{MSupp}_R(T/R)=\{M\}\}$. 
 We are able to give dual results with $\mathcal T^*:=\{T\in[R,S]\mid|\mathrm{MSupp}_R (S/T)|= 1\}$, but they do not appear in this paper because they are not used in the sequel.  

\begin{proposition}\label{6.1} A subextension $R\subset T$ of an  FCP extension $R\subset S$  is $M$-crucial when $T\in {\mathcal T}_ M$ is  such that $\mathrm{Supp}(T/R)\subseteq \mathrm{Max}(R)$. Moreover, $\mathcal A$ is
the set of all minimal elements of $\mathcal T$. 
  
 In case  $R\subset S$ has FIP,  ${\mathcal T}_ M$ has a greatest element $s(M):=\prod_{T\in\mathcal T_M}T$.
\end{proposition} 

\begin{proof} Let $T\in {\mathcal T}_ M$ 
 be such that $\mathrm{Supp}(T/R)\subseteq \mathrm{Max}(R)$.
 It follows that  $\mathrm{Supp}_R(T/R)=\mathrm{MSupp}(T/R)=\{M\}$ and $R_P=T_P$ for each $P\in\mathrm{Spec}(R)\setminus\{M\}$. 
 Hence $R\subset T$ is $M$-crucial.
If $U\in[R,S]$ is an atom,  $R\subset U$ is  minimal  and $M:={\mathcal C}(R,U)$ with $\mathrm{MSupp}(U/R)=\{M\}$, whence $U\in\mathcal T$. Since $R\subset U$ is  minimal, $U$ is a minimal element of $\mathcal T$. Conversely, let $U$ be  a minimal element of $\mathcal T$. If  $U$ is not an atom,  there is $U'\in [R,S]$ with $R\subset U'\subset U$. Then, $\emptyset\neq\mathrm{MSupp}(U'/R)\subseteq\mathrm{MSupp}(U/R)=\{M\}$ implies $\mathrm {MSupp}(U'/R)=\{M\}$ giving $U'\in\mathcal T$, contradicting  the minimality of $U$ in $\mathcal T$. Therefore, $U$ is an atom. 

 Since $s(M):=\prod_{T\in\mathcal T_M}T$, we get that $T\subseteq s(M)$ for each $T\in\mathcal T_M$ and $R\subset s(M)$, whence  $\mathrm{MSupp}(s(M)/R)\neq \emptyset$. Since $R_P=T_P$ for each $T\in\mathcal T_M$ and $P\in \mathrm{Spec}(R)\setminus \{M\}$, we get $s(M)_P=R_P$, so that $\mathrm{MSupp}(s(M)/R)=\{M\}$ and $s(M)$ is the greatest element of $\mathcal T_M$. 
\end{proof}

 \begin{proposition}\label{6.111} Let $R\subset S$ be an FCP atomistic extension such that   $|\mathcal A|<\infty$, with $\mathrm{MSupp}(S/R)=\{M_1,\ldots,M_n\}$. Set $\mathcal A_M=\mathcal A\cap {\mathcal T}_ M$ and, for each $k\in\mathbb N_n$, $V_k:=\prod_{i=1}^ks(M_i)$ and $V_0:=R$.
 
  Then,  the following statements  hold.   
\begin{enumerate}
\item $R\subset S$ has FIP.

\item For each  $M\in\mathrm{MSupp}(S/R),\  s(M)=\prod_{A\in \mathcal A_M}A$.

\item For each $k\in\mathbb N_n,\ \mathrm{Supp}(V_k/V_{k-1})=\{M_k\}$ and $\{V_k\}_{k=0}^n$ is an increasing chain such that $V_n=S$.

\item For each $T\in]R,S]$, there exists $T_i\in {\mathcal T}_ {M_i}$, for each $M_i\in\mathrm{MSupp}(T/R)$, such that $T=\prod_{M_i\in\mathrm{MSupp}(T/R)}T_i$.

\item $R\subset S$  is almost-Pr\"ufer. 
\end{enumerate}
\end{proposition}  

 \begin{proof} (1) Since $|\mathcal A|<\infty$ and any element of $[R,S]$ is a product of atoms, then $|[R,S]|<\infty$ and $R\subset S$ has FIP.

(2) Let $M\in\mathrm{MSupp}(S/R)$. Any element of ${\mathcal T}_ M$ is a product of atoms, which are necessarily in ${\mathcal T}_ M$, and then in ${\mathcal A}_ M$. It follows from Proposition \ref{6.1} that $s(M)=\prod_{A\in \mathcal A_M}A$.

(3)  Let $M_j\in \mathrm{MSupp}(S/R)$ for some $j\in\mathbb N_n$. Assume $1<k<n$. Since $V_k=\prod_{i=1}^ks(M_i)$ and $V_{k-1}=\prod_{i=1}^{k-1}s(M_i)$, we have $(V_{k-1})_{M_j}=(V_k)_{M_j}=R_{M_j}$ if $j>k$, so that $M_j\not\in\mathrm{Supp}(V_k/V_{k-1})$. If $j=k>k-1$, then, $(V_{k-1})_{M_k}=R_{M_k}$ and $(V_k)_{M_k}=s(M_k)_{M_k}\neq R_{M_k}$, so that $M_k\in \mathrm{Supp}(V_k/V_{k-1})$. At last, if $j<k$, then, $(V_{k-1})_{M_j}=(V_k)_{M_j}=s(M_j)_{M_j}$, so that $M_j\not\in\mathrm{Supp}(V_k/V_{k-1})$. Hence, $\mathrm{Supp}(V_k/V_{k-1})= \{M_k\}$. If $k\in\{1,n\}$, the same reasoning holds. The end of (3) is obvious.

(4) Let $T\in]R,S]$. Let $I\subset \mathbb N_n$ be such that   $\mathrm{Supp}(T/R)=\{M_i\mid i\in I\}$. Since $T$ is a product of atoms $A_{\alpha}$, for each $i\in I$, set $T_i=\prod[A_{\alpha}\mid A_{\alpha}\in  \mathcal T_{M_i}]$. Then, $T_i\in   \mathcal T_{M_i}$ and  $T=\prod_{M_i\in\mathrm{MSupp}(T/R)}T_i$.

(5) Since  $R\subset A$ is minimal for any $A\in \mathcal A$, either $R\subset A$ is integral or $R\subset A$ is integrally closed. Moreover, by \cite[Lemma 1.5]{Pic 6}, for a given $M\in\mathrm{MSupp}(S/R)$, minimal extensions $R\subset A$, for $A\in \mathcal T_M$, are either all integral, or all integrally closed. Reorder $\mathrm{MSupp}(S/R)$ such that for some $k\in \mathbb N_n,\ R\subset A$ is integrally closed for all $A\in \mathcal T_{M_i}$ and for any $i\leq k$ and $R\subset A$ is integral for all $A\in \mathcal T_{M_i}$ and for any $i> k$. Then, $R\subset V_k$ is Pr\"ufer and $V_k\subset S$ is integral, so that $R\subset S$ is almost-Pr\"ufer.
\end{proof}

\begin{proposition}\label{6.2} Let $R\subset S$ be an FCP extension, such that $(\mathcal T,\subseteq)$ is a tree. Then, 
\begin{enumerate}
\item The elements of $\mathcal T$ are $\Pi$-irreducible in $[R,S]$.

\item For each  $M\in\mathrm{MSupp}(S/R),\ {\mathcal T}_M$ is a chain,  whose least element $i(M)$  is the only $T\in[R,S]$ satisfying  $R\subset T$ is minimal and  $M={\mathcal C}(R,T)$. 

\item Let $M\in\mathrm{MSupp}(S/R)$. For each $T\in[R,S]$ such that $M\in\mathrm{MSupp}(T/R)$, we have $i(M)\subseteq T$.

\end{enumerate}
\end{proposition} 

\begin{proof} (1)  Let $T\in\mathcal T$ and $M\in\mathrm{MSupp}(S/R)$ such that $\mathrm{MSupp}(T/R)=\{M\}$. It follows that $R_M\neq T_M$ and $R_{M'}=T_{M'}$ for each $M'\in\mathrm{Max}(R)\setminus\{M\}$. Let $T_1,T_2\in[R,S]$ be such that $T=T_1T_2$. Then we have $R\subseteq T_i\subseteq T$ for $i\in \mathbb{N}_2$, giving $(T_i)_{M'}=R_{M'}=T_{M'}$ for each $ M'\in\mathrm{Max}(R)\setminus\{M\}$, and $R_M\subseteq(T_i)_M\subseteq T_M$ for $i\in\mathbb{N}_2$, giving, for each $i\in\mathbb{N}_2$, either $R=T_i$ (a), or $\mathrm{MSupp}(T_i/R) =\{M\}$ (b). Fix $i$ and let $j\in\mathbb{N}_2\setminus\{i\}$. Case (a) gives $T=T_j$. Case (b) gives that $T_i\in\mathcal T$. In this case, either $T_j=R$, giving $T=T_i$, or $T _j\in\mathcal T$. Hence $T_i$ and $T_j$ are comparable, because $\mathcal T$ is a tree. Therefore, $T$ is equal to the greatest element of $\{T_1,T_ 2\}$ and $T$ is $\Pi$-irreducible. 

(2) Let $M\in\mathrm{MSupp}(S/R)$ and $T_1,T_2\in{\mathcal T}_M$. Set $U:=T_1T _2$, so that $T_i\subseteq U$ for $i\in\mathbb{N}_2$ and then $(T_i)_{M'}=R_{M'}=U_{M'}$ for each $M'\in\mathrm{Max}(R)\setminus \{M\}$, and $R_M\subset(T_i)_M\subseteq U_M$. Then, $\mathrm{MSupp}(U/R)=\{M\}$. Therefore,  $U\in\mathcal T$, so that $T_1$ and $T_2$ are comparable and ${\mathcal T}_M$ is a chain.

 By Lemma \ref{1.10}, there is $T\in[R,S]$ such that $R\subset T$ is minimal with $M=\mathcal C(R,T)$. Obviously, we have $\mathrm{MSupp}(T/R)=\{M\}$, so that 
 $T\in{\mathcal T}_M$. Since ${\mathcal T}_M$ is a chain, $T$ is the least element of ${\mathcal T}_M$, because any $U\in{\mathcal T}_M$ is comparable to $T$, and we cannot have $U\subset T$. Moreover, since any $T'\in[R,S]$ such that $R\subset T'$ is minimal with $M=\mathcal C(R,T')$ satisfies $T'\in{\mathcal T}_M$, we have just proved that there is only one $T'\in[R,S]$ such that $R\subset T'$ is minimal with $M=\mathcal C(R,T')$.
We set $i(M):=T'$. 

(3) Let $T\in[R,S]$ with $M\in\mathrm{MSupp}(T/R)$. For each $M'\in\mathrm{Max}(R) \setminus \{M\}$, we have $R_{M'}=i(M)_{M'}\subseteq T_ {M'}$ and a  maximal chain  $R=R_0 \subset\cdots\subset R_i\subset\cdots\subset R_n=T$, where $R_{i-1}\subset R_i$ is minimal for each $i\in\mathbb{N}_n$. Since $M\in\mathrm{MSupp}(T/R)$, there is some $i\in\mathbb{N}_n$ such that $M=\mathcal C(R_{i-1},R_i)\cap R$,  by Lemma \ref{1.9}. Let $k\in\mathbb{N}_n$ be the least $i$ satisfying this property. Assume $k\neq 1$, so that $M\not\in\mathrm{MSupp}(R_{k-1}/R)$.  Again by Lemma \ref{1.10}, there is $R_1'\in[R,R_k]$ such that $R\subset R'_1$ is minimal with $M=\mathcal C(R,R'_1)$. From (2), we deduce that $R'_1 =i(M)$, so that $i(M)\subseteq T$. If $k=1$, then $R_{k-1}=R$ and $R_k=i(M)\subseteq T$. 
\end{proof}

We will see in Proposition \ref{6.6}, that, under some conditions, for an FCP extension $R\subset S$, any element of $[R,S]$ can be written in a unique way, as a product of elements of $\mathcal T$. But Remark \ref{6.8} shows that this property does not always hold.

 Now, we look at some properties of  arithmetic   FCP extensions. 

\begin{theorem}\label{6.4} Let $R\subset S$ be an arithmetic FCP extension. Then, 
\begin{enumerate}
\item $R\subset S$ has FIP and $|[R,S]|\leq \prod_{M\in\mathrm{MSupp}(S/R)}(1+\ell[R_M,S_M])$.

\item $\mathcal T_M$ 
is a chain for each  $M\in\mathrm{MSupp}(S/R)$ 
and   $(\mathcal T,\subseteq)$
is treed.
\end{enumerate}
\end{theorem} 

\begin{proof}  Clearly,   $R_M\subset S_M$ has FCP for each $M\in\mathrm{MSupp}(S/R)$    (\cite[Proposition 3.7(a)]{DPP2}). 

(1) Therefore, $|[R_M,S_M]|<\infty$ follows, since $[R_M,S_M]$ is a chain  and $R_M\subset S_M$ has FIP for each $M\in\mathrm{MSupp}(S/R)$. Hence, $R\subset S$ has FIP \cite[Proposition 3.7(b)]{DPP2}. By \cite[Theorem 3.6 (a)]{DPP2}, we have $|[R,S]|\leq\prod_{M\in\mathrm{MSupp}(S/R)}|[R_M,S_M]|$. But $|[R_M,S_M]|=1+\ell[R_M,S_M]$ holds since $[R_M,S_M]$ is la chain, giving the requested  inequality. 

(2) Let $T_1,T_2\in{\mathcal T_M}$, 
  so that $\{M\}=\mathrm{MSupp}(T_i/R)$ 
 
   for $i\in\mathbb{N}_2$. Then, $(T_ 1)_{M'}=(T_2)_{M'}=R_{M'}$ 
    for each $M'\in\mathrm{MSupp}(S/R) \setminus \{M\}$. Moreover, $(T_i)_M\in[R_M,S_M]$  for $i\in\mathbb{N}_2$, so that 
$(T_1)_M$ and $(T_2)_M$ are comparable, and so are $T_1$ and $T_2$. 
 Then, $\mathcal T_M$ is a chain.

Let $T_1,T_2,T\in{\mathcal T}$ be such that $T_i\subseteq T$ for $i\in\mathbb{N}_2$ and $M\in\mathrm{MSupp}(S/R)$  such that $\mathrm{MSupp}(T/R)=\{M\}$. We get, for $i\in\mathbb{N}_2$, that $\emptyset\neq\mathrm{MSupp}(T_i/R)\subseteq\mathrm {MSupp}(T/R)=\{M\}$, so that $\mathrm{MSupp}(T_i/R)=\{M\}$.
It follows that $T_1,T_2\in{\mathcal T}_M$, which is a chain, and then $T_1$ and $T_2$ are comparable. Therefore,  $\mathcal T$ is a tree.
\end{proof}

 Let $R\subset S$ be an FCP extension and $\mathrm{MSupp}(S/R):=\{M_1,\ldots,M _n\} $. We will use the maps $\varphi:[R,S]\to{\prod}_{i=1}^n[R_{M_i},S_{M_i}]$ defined by $\varphi(T):=(T_{M_1},\ldots,T_{M_n})$ and ${\varphi}_M:[R,S]\to[R_M,S_ M]$ defined by ${\varphi}_M(T):=T_M$, for each $M\in\mathrm{MSupp}(S/R)$. Then $\varphi$ is injective  \cite[Theorem 3.6]{DPP2}. If $\varphi$ is bijective,  $R\subseteq S $ is called a  {\it $\mathcal B$-extension} ($\mathcal B$ stands for bijective). 
 
 \begin{proposition}\label{6.5} An FCP extension  $R \subseteq S$ is a $\mathcal B$-extension if and only if $R/P$ is  local for each $P\in\mathrm{Supp}(S/R)$. 

The above ``local" condition on the factor domains $R/P$ holds in case $\mathrm {Supp}(S/R)\subseteq\mathrm {Max}(R)$, and, in particular, if $R\subset S$ is  integral.
\end{proposition}

\begin{proof} One implication appears in the proof of \cite[Theorem 3.6(b)]{DPP2} which uses \cite[Lemma 3.5]{DPP2}. 

Conversely, assume that $\varphi$ is  bijective and  that there is some $P\in\mathrm{Supp}(S/R)$ contained in two elements  $M_1,M_2\in\mathrm {Max}(R)$. Consider a maximal chain ${\mathcal C}:R=R_0\subset\cdots\subset R_i\subset\cdots\subset R_n=S$, where $R_{i-1}\subset R_i$ is minimal for each $i\in\mathbb{N}_n$. Since $P\in\mathrm{Supp}(S/R)$, there exists some $i\in\mathbb{N}_n$ such that $P=N\cap R$, where $N:=\mathcal C(R_{i-1},R_i)$ by Lemma \ref{1.9}. In particular, $P\in\mathrm{Supp}(R_i/R)$, which implies that $M_j\in\mathrm{Supp}(R_i/R)$ for $j\in\mathbb{N}_2$. Since $\varphi$ is surjective, there is $T\in[R,S]$ such that $T_{M_1}=(R_i)_{M_1}$ and $T_M=R_M$ for each $M\in\mathrm{MSupp}(S/R)\setminus\{M_1\}$. In particular, $M_2\not\in\mathrm{Supp}(T/R)$. Localizing ${\mathcal C}$ at $M_1$, we get that 
$P_{M_1}\in\mathrm{Supp}((R_i)_{M_1}/R_{M_1})=\mathrm{Supp}(T_{M_1}/R_ {M_1})$ so that $P\in\mathrm{Supp}(T/R)$, Now $M_2\in\mathrm{Supp}(T/R)$ is absurd. Then, $R/P$ is a local ring for each $P\in\mathrm {Supp}(S/R)$. 

If $\mathrm{Supp}(S/R)\subseteq\mathrm{Max}(R)$, it follows from Lemma \ref{1.9} that each $\mathcal C(R_{i-1},R_i)$ lies over a maximal ideal of $R$ (and so $R/P$ is a field, hence local for each $P\in\mathrm{Supp}(S/R)$). Finally, it is standard that maximal ideals lie over maximal ideals in any integral extension (cf. \cite[Theorem 44]{K}).
\end{proof} 

 Under an additional assumption, next proposition gives a converse to Theorem \ref{6.4}. Set $\mathrm{MSupp}(S/R) :=\{M_1,\ldots ,M_n\}$ for an FCP extension $R\subset S$. It  follows that the elements  of $\mathrm{MSupp}(T/R)$ are some $M_i$ when $T\in ]R,S]$.

\begin{proposition}\label{6.6} Let $R\subset S$ be an FCP $\mathcal B$-extension.
\begin{enumerate}
\item For each $M\in\mathrm{MSupp}(S/R)$, we have $[R_M,S_M]={\varphi}_M (\mathcal T)\cup\{R_M\}={\varphi}_M(\mathcal T_M)\cup\{R_M\}$.

\item Any  $T\in]R,S]$ is a product of  $|\mathrm{MSupp}(T/R)|$ distinct elements $E_i\in\mathcal T$ in a unique way such that $E_i\in\mathcal T_{M_i}$, for each $M_i\in\mathrm{MSupp}(T/R)$.

\item $R\subset S$ is arithmetic if and only if $\mathcal T$ 
 is a tree.

\item Assume that $R\subset S$ is arithmetic. Then, $\mathcal T$ 
  is the set of $\Pi$-irreducible 
   elements of $]R,S]$.

   Moreover, $|[R,S]|=\prod_{M\in\mathrm{MSupp}(S/R)}(1+\ell[R_M,S_M])$.
\end{enumerate}
\end{proposition}

\begin{proof}   (1) The following inclusions are obvious: ${\varphi}_M(\mathcal T_M)\cup\{R_M\}\subseteq{\varphi}_M(\mathcal T)\cup\{R_M\}\subseteq[R_M,S_M]$. Let $E\in]R_M,S _M]$. Since $\varphi$ is a bijection, there exists $T\in]R,S]$ such that $T_M=E$ and $T_{M'}=R_{M'}$ for each $M'\in\mathrm{MSupp}(S/R)\setminus\{M\}$. The  result follows from $E={\varphi}_M(T)$ and $T\in\mathcal T_M$.

(2) Let $T\in]R,S]$. 
 From $\mathrm{MSupp}(T/R)\neq\emptyset$, we infer that some $M_i\in\mathrm{MSupp}(T/R)$. 
Since $\varphi$ is a bijection, 
there exists a unique $E_i\in[R,S]$ such that $(E_i)_{M_i}=T_{M_i}$ and $(E_i)_M=R_M$ for $M\neq M_i$.  It follows that $E_i\in\mathcal T_{M_i}$  for each $M_i\in\mathrm{MSupp}(T/R)$. 
  Set $E_j=R$ when $M_j\not\in\mathrm{MSupp}(T/R)$ and
 $E:=E_1\cdots E_n$. For each $j\in\mathbb N_n$, we get that $E_{M_j}=T_{M_j}$, 
 so that $E=T$ is   the
product of  $|\mathrm{MSupp}(T/R)|$ distinct elements of $\mathcal T$. The uniqueness of these elements is obvious.

(3) One implication is Theorem \ref{6.4}. Assume that $\mathcal T$ 
is a tree. Let $M\in\mathrm{MSupp}(S/R)$. By Proposition \ref{6.2}, 
  $\mathcal T_M$ 
    is a chain. Since ${\varphi}_M$ preserves order, (1) implies  that $[R_M,S_M]$ is a chain. 

(4) Assume that $R\subset S$ is arithmetic. Then, $\mathcal T$ 
  is a tree. It results from Proposition \ref{6.2} 
  that the elements of $\mathcal T$ 
   are $\Pi$-irreducible. 
    Conversely, let $T\in ]R,S]$ be $\Pi$-irreducible. In view of  (2), $T$ is a product of elements $E_i\in\mathcal T_{M_i}$, where $M_i\in\mathrm {MSupp}(T/R)$, and then, of only one element of $\mathcal T$, so that $T\in\mathcal T$.

Now  $\prod_{M\in\mathrm{MSupp}(S/R)}|[R_M,S_M]|=\prod_{M\in\mathrm{MSupp}(S/R)}(1+\ell[R_M,S_M]) = |[R,S]|$, since $\varphi$ is  bijective and each $R_M\subset S_M$  a chain.
\end{proof}

 For the definition of the Goldie dimension of a 
 distributive lattice, the reader may look at \cite[p. 14 and Exercise 8, p. 33]{NO}. We will use the following results.

\begin{proposition} \label{5.6}\cite[Theorem 1.5.9]{NO} If $R\subseteq S$ is an FCP arithmetic extension, its Goldie dimension  is the integer $n$   such that $R =B_1\cap\cdots\cap B_n$ is an irredundant representation of $R$ by $\cap$-irreducible elements. 
\end{proposition}

 \begin{proposition}\label{6.7} Let $R\subset S$ be an FCP arithmetic $\mathcal B$-extension extension. Then, the Goldie dimension of $[R,S]$ is  $n:=|\mathrm{MSupp}_R(S/R)|$. 
\end{proposition}

\begin{proof}  For each $M_i\in\mathrm{MSupp}(S/R)$ there exists a unique $E'_i\in[R,S]$ such that $(E'_i)_{M_i}=R_{M_i}$ and $(E'_i)_M=S_M$ for $M\neq M_i$. 
 Set $E ':=E'_1\cap\ldots\cap E'_n$. For each $i$, we get that $(E'_i)_{M_i}=R_{M_i}$, so that $E'=R$ is  the intersection of  $|\mathrm{MSupp}(S/T)|$ distinct elements of $[R,S]$. The uniqueness of these elements is obvious. 
 
 We claim that the $E'_i$ are $\cap$-irreducible. Let $T,T'\in[R,S]$ be such that $E'_i=T\cap T'$. Then we have $R\subset E'_i\subseteq T,T'$, giving $(E'_i)_{M}=S_{M}=T_{M}=T'_{M}\ (*)$ for each $ M\in\mathrm{Max}(R)\setminus\{M_i\}$, and $R_{M_i}=(E'_i)_{M_i}\subseteq T_{M_i},T'_{M_i}$. Since $R\subset S$ is arithmetic, $T_{M_i}$ and $T'_{M_i}$ are comparable, and so are $T$ and $T'$ by $(*)$. It follows that $E'_i$ is the least element of $\{T,T'\}$, and then is $\cap$-irreducible
so that  
 $n$ is the Goldie dimension of $[R,S]$ by Proposition \ref{5.6}.
\end{proof} 

\begin{remark}\label{6.8} Using \cite[Remark 3.4(b)]{DPP2}, we exhibit   a non  $\mathcal B$-extension,  for which  some statements  of Proposition \ref{6.6} do not hold. 

We now summarize the context  of the above quoted remark.

Let $R$ be a two-dimensional Pr\"ufer domain  with exactly two height-$2$ maximal ideals, $N$ and $N '$, each of which containing the unique height $1$ prime ideal $P$ of $R$. Set $R_0:=R,\ R_1:=R_N,\ R'_1:=R_{N'},\ R_2:=R_P$ and  $R_3:=K$, the quotient field of $R$. Since each overring of a Pr\"ufer domain is an intersection of localizations \cite[Theorem 26.2]{MIT}, it is easy to check that $R_0 \subset R_1,\, R_0\subset R'_1,\,R_1\subset R_2,\,R'_1\subset R_2$ and $R_2\subset R_3$ are Pr\"ufer minimal extensions. 

Moreover, $\mathcal{C}(R_0,R_1)=N',\, \mathcal{C}(R_ 1,R_2)=NR_N$ and $\mathcal{C}(R_2,R_3)=P$, which is an ideal of $R_2$ because $P=(R_0:R_2)$. It follows that $\mathrm{Supp}(K/R)=\{P,N,N'\}$ and $\mathrm{MSupp}(K/R)=\{N,N'\}$. By Proposition \ref{6.5}, the map $\varphi$  is not bijective, since $R/P$ is not local. The poset $\mathcal T=\{R_1,R'_1\}$ is a  tree. Moreover, $[R_N,K_N]=\{R_N,R_P,K\}=\{R _1,R_2,R_3\}$ and $[R_{N'},K_{N'}]=\{R_{N'},R_P,K\}=\{R'_1,R_2,R_3\}$ are chains. However, ${\varphi}_N(\mathcal T)=\{R_1,R_2\}$, because $(R_{N'})_N=R_ P$. Then, Proposition \ref{6.6}(1) is not satisfied. In the same way, Proposition \ref{6.6}(2) is not satisfied, because $K$ is not a product of elements of $\mathcal T$. In particular,  some $\Pi$-irrreducible element as $R_3=K$ of $[R,S]$ is not in $\mathcal T$. 
\end{remark}

The conditions of Proposition \ref{6.6} hold in the following context and provide us a structure of Boolean lattices in Proposition \ref{4.0201}.

\begin{proposition}\label{6.9} Let $R\subset S$ be an FCP extension. Assume that $\mathrm{Supp}(T/R)\cap\mathrm{Supp}(S/T)=\emptyset$ for all $T\in [R,S]$. Then:

\begin{enumerate}
\item $\mathrm{Supp}(S/R)\subseteq\mathrm{Max}(R)$ and $R\subseteq S$ is a $\mathcal B$-extension.

\item $\mathcal T$ 
 is an antichain. 
\end{enumerate}
\end{proposition}

\begin{proof} (1) Assume that $\mathrm{Supp}(T/R)\cap\mathrm{Supp}(S/T)=\emptyset$ for all $T\in[R,S]$. We use Proposition \ref{6.5} to show that $\varphi$ is a bijection. Let $P\in\mathrm{Supp}(S/R)$. Consider a maximal chain ${\mathcal C}:R= R_0\subset\cdots\subset R_i\subset\cdots\subset R_n=S$, where $R_{i-1}\subset R_i$ is minimal for each $i\in\mathbb{N}_n$. Since $P\in\mathrm{Supp}(S/R)$, there exists some $k\in\mathbb{N}_n$ such that $P=N\cap R$, where $N:=\mathcal C(R_ {k-1},R _k)$, in view of Lemma \ref{1.9}. In particular, $P\in\mathrm{Supp}(R_k/R)$, and, more precisely, $P\in\mathrm{Supp}(R_k/R_{k-1})\subseteq\mathrm{Supp}(S/R_{k-1})$. Assume that $P$ is not a maximal ideal, and let $M\in\mathrm{Max}(R)$ be such that $P\subset M$. Then, $M\in\mathrm{Supp}(R_k/R)$. Moreover, there is some $j< k$ such that $M=N'\cap R$, where $N':=\mathcal C(R_{j-1},R_j)$. Indeed, $j\neq k$ because $N\cap R=P\neq M$. Therefore  $M\in\mathrm{Supp}(R_j/R)\subseteq\mathrm{Supp}(R_{k-1}/R)$, since $j \leq k-1$ and $M\in\mathrm{Supp}(R_{k-1}/R)\cap\mathrm{Supp}(S/R_{k-1})$ is absurd. Hence, $\mathrm{Supp}(S/R)\subseteq\mathrm{Max}(R)$, and  $\varphi$ is a bijection  by Proposition \ref{6.5}.

(2) Let $T,T'\in\mathcal T$ be such that $T'\subset T$ and $M$ 
 $\in\mathrm{Max}(R)$  such that $\mathrm{Supp}(T/R)=\{M\}$. 
  Since $T'\subset T$, we get that $\mathrm{Supp}(T'/R)\subseteq\mathrm{Supp}(T/R)$, giving $\mathrm {Supp}(T'/R)=\{M\}$. But, $\emptyset\neq\mathrm{Supp}(T/T')\subseteq\mathrm{Supp}(T/R)=\{M\}$ implies that $\{M\}=\mathrm{Supp}(T/T')\subseteq\mathrm{Supp}(S/T')$. Therefore, $M\in\mathrm{Supp}(T'/R)\cap\mathrm{Supp}(S/T')=\emptyset$, a contradiction. Then,  two distinct elements of $\mathcal T$ are incomparable, and $\mathcal T$ is an antichain.  
  \end{proof}

Let $R\subset S$ be an FCP extension. In the following, we will meet the condition that $\mathrm {Supp}(T/R)\cap\mathrm{Supp}(S/T)=\emptyset$ for some $T\in[R,S]$. Here is a theorem which gives a stronger result.

\begin{theorem}\label{6.10} If $R\subset S$ has FCP, then $\mathrm {Supp}(T/R)\cap\mathrm{Supp}(S/T)=\emptyset$ for all $T\in[R,S]$ if and only if  
 $R\subset S$ is locally minimal. 
 In this case, $R\subset S$ is an FIP factorial extension. 
\end{theorem}

\begin{proof}  Assume that $\mathrm{Supp}(T/R)\cap\mathrm{Supp}(S/T)=\emptyset$ for all $T\in[R,S]$. By Propositions \ref{6.9} and \ref{6.6}, $R\subset S$ is arithmetic. Let $T\in[R,S]$, there is $U\in[R,S]$ such that $U\cap T=R$ and $UT=S$ \cite[Lemma 3.7]{Pic 3}, which gives $U_M\cap T_M=R_M$ and $U_MT_M=S_M$ for each $M\in\mathrm{MSupp}(S/R)$, so that $\{U_M,T_M\}=\{R_M,S_M\}$, whence,  $R_M\subset S_M$ is  minimal. 

Conversely, if there exists some $N\in\mathrm {MSupp}(T/R)\cap\mathrm{MSupp}(S/T)$ for some $T\in [R,S]$, 
 we get that $T_N\neq R_N,S_N$, so that $|[R_N,S_N]|\geq 2$. But $R_N\subset S_N$ is minimal,
 a contradiction.

If these conditions hold, then $R\subseteq S$ has FIP by Theorem \ref{6.4}, and is factorial by Propositions \ref{6.6} and \ref{6.1}. 
\end{proof}

 \section {Boolean FCP extensions}

\subsection { General properties of Boolean extensions } 
 Let $R \subseteq S$ be an extension and $T\in [R,S]$. Then, $T'\in [R,S]$ is called a {\it complement} of $T$ if $T\cap T'=R$ and $TT'=S$. If  $R \subseteq S$ is    distributive, then $T$ has at most one complement \cite[Exercise 9, page 33]{NO}. We denote this complement by $T^{\circ}$ when it exists.  
 
We recall that $R\subseteq S$ is Boolean if and only if $R\subseteq S$ is distributive and each $T\in[R,S]$ has a (unique) complement $T^{\circ}$ \cite[Definition page 129]{R}.  
In a Boolean FIP extension $R\subseteq S$, any $T\in]R,S[$ is, in a unique way, a product of finitely many atoms and the intersection of finitely many co-atoms. To see this, use Lemma  \ref{5.7}, \cite[Theorems 5.1 and 6.3]{R}, and the fact that $[R,S]$ is a complete Boolean lattice. In particular, if $A$ is an atom, then $A^{\circ}$ is a co-atom \cite[Theorem 3.43]{R}.  Next Theorem  characterizes  Boolean FCP extensions amid  distributive FCP extensions:

\begin{theorem}\label{4.0} Let $R\subseteq S$ be an FCP distributive extension and set $n:=\ell[R,S]$. Then, the following conditions are equivalent:
\begin{enumerate}
\item $R\subseteq S$ is a Boolean extension. 

\item $n=|\mathcal A|$, where $\mathcal A$ is the set of atoms of $[R,S]$.

\item $|[R,S]|=2^n$.
\end{enumerate}
If these conditions hold, then $R\subseteq S$ is FIP factorial 
 and a simple pair.
Moreover, for each $T,U\in[R,S]$ with $T\subseteq U$, all maximal chains of  $[T,U]$ have the same length.

\end{theorem}
\begin{proof} To begin with, $R\subseteq S$ has FIP by Lemma \ref{5.7}. Moreover, $[R,S]$ has exactly $n\ \Pi$-irreducible element by Lemma \ref{5.7}. Since any atom is $\Pi$-irreducible, the equivalences of (1), (2) and (3) are a translation of the equivalences given in \cite[page 292]{St}.

If   these conditions hold, then $R\subseteq S$ is factorial  and a simple pair by Theorem \ref{5.8}.
Since a Boolean lattice is a distributive lattice, the last results comes from Proposition  \ref{1.0}. 
\end{proof} 

\begin{example}\label{4.01}  (1) An  extension  $R\subset S$ is Boolean and   chained   if and only if it is  minimal. One implication is obvious. Conversely, assume that   $[R,S]$ is  chained and a  Boolean lattice. If  there is some $T\in]R,S[$, it has a complement $T'$. Then, $T\cap T'=R$ and $TT'=S$ implies $\{T,T'\}=\{R,S\} $ since $T$ and $T'$ are comparable, a contradiction.

 (2)  Let $R\subset S$ be a Boolean FCP extension and    $x\in S\setminus R$.  We intend to compute $R[x]^\circ$. Let $\mathcal CA:=\{B_1,\ldots,B_n\}$.
Then $R[x]=\cap [B_{\alpha}\in Y]$  for a unique family $Y:=\{B_{\alpha}\}$ of co-atoms (\cite[Theorem 6.3]{R}), so that $x\in B_{\alpha}$ for each $B_{\alpha}\in Y$. Let $B_{\beta}\in \mathcal CA$  which does not contain $x$.   Then  $R[x]\not\subseteq B_{\beta}$
and $Y$ is the set of co-atoms containing $x$.   Now, $(R[x])^{\circ}=(\cap [B_{\alpha}\in Y])^{\circ}=\Pi[(B_{\alpha})^{\circ}\mid B_{\alpha}\in Y]$, where the $(B_{\alpha})^{\circ}$ are atoms. Then, $(R[x])^{\circ}$ is the product of atoms which are complements of the co-atoms containing $x$.  But we also have $(R[x])^{\circ}= \cap [B_{\beta}\in \mathcal CA\setminus Y]$. Indeed, set $T:= \cap [B_{\beta}\in \mathcal CA\setminus Y]$. Obviously, $R[x]\cap T=R$ and assume that $R[x]T\neq S$, so that $R[x]T=\cap [B_{\gamma}\in X]$ for some $X\subseteq \mathcal CA$. Let $B_{\gamma}\in X$. Then, $R[x]=\cap [B_{\alpha}\in Y]\subseteq B_{\gamma}$ implies $B_{\alpha}=B_{\gamma}$ for some $B_{\alpha}\in Y$. In the same way, $B_{\beta}=B_{\gamma}$ for some $B_{\beta}\in \mathcal CA\setminus Y$, a contradiction. Then, $R[x]T= S$ and $T=(R[x])^{\circ}$.

(3) By Theorem \ref{4.0}, an FCP Boolean extension $R\subset S$ verifes $|[R,S]|=2^n$, where $n=\ell[R,S]$. But an extension $R\subseteq S$ may be distributive with $|[R,S]|=2^n$ for some integer $n$ without being Boolean. It is enough to  consider a chained extension   of length $2^n-1$.
\end{example}

The following lemma is needed for the next proposition. 
 See the close notion of {\it patching} due to Dobbs-Shapiro \cite {DS1}.

\begin{lemma}\label{4.001} Let $R\subseteq S$ be an FCP extension and $M\in\mathrm{MSupp}(S/R)$. For any $T'\in[R_M,S_M]$ such that $R_M\subset T'$ is minimal, there exists $T\in[R,S]$ such that $R\subset T$ is minimal with $T_M=T'$. 
\end{lemma}
\begin{proof} Let $\varphi:S\to S_M$ be the canonical ring morphism and set $T'':=\varphi ^{-1}(T')$. Then $T''\in[R,S]$ is such that $T'=T''_M\neq R_M$, so that $M\in\mathrm {MSupp}(T''/R)$. 
 From Lemma \ref{1.10} we deduce  the existence  of some $T\in[R,T'']\subseteq[R,S]$ such that $R\subset T$ is minimal with $M=\mathcal C(R,T)$. Hence, $R_M\subset T_M\subseteq T''_M=T'$ gives $T_M=T'$. 
\end{proof}

\begin{proposition}\label{4.02} Let $R\subseteq S$ be an FCP extension. The following statements are equivalent: 
\begin{enumerate}
\item $R\subseteq  S$ is  Boolean.

\item $R_M \subseteq S_M$ is  Boolean for each $M\in\mathrm{MSupp}(S/R)$.

\item $R_P \subseteq S_P$ is  Boolean  for each $P\in\mathrm{Supp}(S/R)$.

\item $R/I\subseteq S/I$ is  Boolean  for each  ideal $I$ shared by $R$ and $S$.

\item  $R/I\subseteq S/I$ is  Boolean  for some  ideal $I$ shared by $R$ and $S$.
\end{enumerate}
\end{proposition}
\begin{proof} We have obviously (1) $\Rightarrow$ (3) $\Rightarrow$ (2) and (1) $\Leftrightarrow$ (4) 
$\Rightarrow$ (5) $\Rightarrow$ (1).
 Conversely, assume that $R_M \subseteq S_M$ is  Boolean  for each $M\in\mathrm{MSupp}(S/R)$. Then, $R_M \subseteq S_M$ is  Boolean for each $M\in\mathrm{Max}(R)$. It follows that the distributivity property holds in $[R,S]$ since it holds in any $[R_M,S_M]$. It remains to show that any $T\in[R,S]$ has a complement. Let $M\in\mathrm{MSupp}(S/R)$. Then, $T_M$ has a complement  $(T_M)^{\circ}$  
 in $[R_M,S_M]$ satisfying $T_M\cap (T_M)^{\circ}=R_M\ (*
)$ and $T_M(T_M)^{\circ}=S _M\ (**)$. Since $[R_M,S_M]$ is Boolean, any of its element is a product of its atoms  \cite[Theorem 5.2]{R}. Then, $(T_M)^{\circ}=\prod_{i\in I_M}R'_{i,M}$, where the $R'_{i,M}$ are atoms of $[R_M,S_M]$, that is $R_M\subset R'_{i,M}$ is minimal. By Lemma \ref{4.001}, for each $R'_{i,M}\in[R_M,S_M]$, there is $R_{i,M}\in [R,S]$ such that $R\subset R_{i,M}$ is minimal, with $R'_{i,M}=(R_{i,M})_M$. In particular, $(R_{i,M})_{M'}=R_{M'}$ for each $M'\in\mathrm{MSupp}(S/R)\setminus\{M\}$. Setting $T':=\prod_{M\in\mathrm{MSupp}(S/R)}(\prod_{i\in I_M}R_{i,M})$, we get, for each $M\in\mathrm{MSupp}(S/R)$ that $T'_M=\prod_{i\in I_M}R'_{i,M}=(T_M)^{\circ}$, so that $(*)$ and $(**)$ give $T_M\cap T'_M=R_M$ and $T_MT'_M= S_M$, and, to end, $T\cap T'=R$ and $TT'=S$ showing that $T'$ is the complement of $T$. Therefore, $R\subseteq S$ is  Boolean.
\end{proof}

 \begin{corollary}\label{5.9} 
  Let $R\subset S$ be an  FIP extension.
 The following statements are equivalent: 
\begin{enumerate}
\item $R\subset  S$ is atomistic and arithmetic;

\item $R \subset S$ is  Boolean and arithmetic;

\item $R \subset S$ is  locally minimal.
\end{enumerate} 

Assume that these conditions hold and  let the 
 set of atoms be $\mathcal A=\{A_1,\ldots,A_a\}$ where $a$ is some integer. Then, the complement of any $T=\prod_{i\in I}A_i\in[R,S]$, where $I\subseteq\mathbb{N}_a$, is $T^{\circ}:=\prod_{j\in J}A_j$ where $J:=\mathbb{N}_a\setminus I$. If in addition $R\subset S$ is integral, then $\ell[R,S]=|\mathrm{MSupp}(S/R)|$. 
\end{corollary} 
\begin{proof}  (1) $\Rightarrow$ (2) by Proposition \ref{5.4} and the equivalences given in \cite[page 292]{St}, 
(2) $\Rightarrow$ (1) by \cite[Theorem 5.2]{R} and 
(2) $\Leftrightarrow$ (3) by Proposition \ref{4.02} and Example \ref{4.01}(1). 

Assume that these conditions hold.
We observe that $S$ is the product of all atoms. 
Let $T =\prod_ {i\in I}A_i\in [R,S]$, where $I\subseteq \mathbb{N}_a$. Since $T^{\circ}$ is a product of atoms, the relations $T\cap T^{\circ}=R$ and $TT^{\circ}=S$ give  
 $T^{\circ}:=\prod_{j\in J}A_j$ where $J: = \mathbb{N}_a \setminus I$. 

 In case $R\subseteq S$ is integral, we use $\ell[R,S]=\sum_{M\in\mathrm{MSupp}(S/R)}\ell[R_M,S_M]$ \cite[Proposition 4.6]{DPP3}. 
\end{proof}

 Next proposition uses the notation of \cite[Proposition 10, p.52]{ALCO}.

\begin{proposition} \label{desc} Let $R\subset S$  be a ring extension, $f: R \to R'$ a  faithfully  flat ring morphism and $S':=  R'\otimes_RS$.  Assume that   $ R '\subset S'$  is distributive. Then,  
\begin{enumerate}
 
 \item  $ R \subset S$  is distributive.
  
 \item Let $T\in[R,S]$ be such that $R'T$ is $\Pi$-irreducible (resp. an atom) in $[R',S']$. Then, $T$ is $\Pi$-irreducible (resp. an atom) in $[R,S]$.
 
\item  In case $R'T$ is $\Pi$-irreducible for any $\Pi$-irreducible $T\in[R,S]$ and  $ R '\subset S'$  is an FIP Boolean extension, so is $ R \subset S$. 
 \end{enumerate}
\end{proposition} 
\begin{proof} 

(1) The ring morphism $\varphi : S \to S'$ defines a map $ \theta: [R',S'] \to [R,S]$ while there is a map  $\psi : [R,S] \to [R',S']$,  defined by $\psi (T)= R'\otimes_RT$  and such that $ \theta \circ \psi   $ is the identity of $[R,S]$ by \cite[Proposition 10, p.52]{ALCO} (it is enough to take  $F=S$ and to observe that if $M$ is an $R$-submodule of $S$, then  with the notation of the above reference,  $R'M$ identifies to  $M\otimes_RR'$). In particular, $\psi$ is injective. The same reference shows that $\psi(T\cap U) = \psi(T) \cap  \psi (U) $ for $U, T\in [R,S]$. It is easy to show that $\psi(TU) = \psi(T)\psi(U)$. If $R'\subset S'$ is distributive, it follows that for $T,U,V\in [R,S]$, we get $\psi[T(U\cap V)]=\psi(T)[\psi(U)\cap\psi(V)]=[\psi(T)\psi(U)]\cap[\psi(T)\psi(V)]=\psi[(TU)\cap(TV)]$, giving $T(U\cap V)=(TU)\cap(TV)$. Then, $R\subset S$ is distributive. 

(2) Let $T\in[R,S]$ be such that $R'T$ is $\Pi$-irreducible  in $[R',S']$ and let $U,V\in[R,S]$ be such that $T=UV$. Then, $R'T=(R'U)(R'V)$, so that either $R'T=R'U$, or $R'T=R'V$, which implies either $T=U$, or $T=V$ and $T$ is $\Pi$-irreducible  in $[R,S]$.

Let $T\in[R,S]$ be such that $R'T$ is  an atom in $[R',S']$.   Assume that $T$ is not an atom in $[R,S]$. There exists $U\in]R,T[$ which yields $R'\subset R'U\subset R'T$, a contradiction. Then, $T$ is an atom in $[R,S]$.

(3) Assume that $R'T$ is $\Pi$-irreducible in $[R',S']$ for any $\Pi$-irreducible $T\in[R,S]$. This holds if $\theta$ (or $\psi$) is bijective. If $R'\subset S'$ is an FIP Boolean extension, any $\Pi$-irreducible element of $[R',S']$ is an atom in $[R',S']$ in view of \cite[page 292]{St}. Moreover, $R\subset S$ is also an FIP  extension since $\theta$ is surjective. Let $T\in[R,S]$ be  $\Pi$-irreducible  in $[R,S]$. Then, $R'T$ is $\Pi$-irreducible in $[R',S']$. Then  $R'T$ is an atom in $[R',S']$, so that $T$ is an atom in $[R,S]$ by (2). Since    $\Pi$-irreducible elements of $[R,S]$ are atoms, the same reference shows that $R\subset S$ is Boolean.
\end{proof}

 \begin{proposition}\label{4.1999} An FCP extension $R\subset S$, whose Nagata extension $R(X)\subset S(X)$ has FIP and is Boolean,  has  FIP  and  is Boolean.
 \end{proposition}

\begin{proof} In view of (\cite[Corollary 3.5]{DPP3}), $R\subset S$ is an FCP extension implies that $S(X)=R(X)\otimes_RS$, so that we can use Proposition \ref{desc} because $R(X)\subset S(X)$ is distributive. Since $R(X)\subset S(X)$ has FIP, the map $\psi:[R,S]\to[R(X),S(X)]$ defined by $\psi(T)=T(X)$ in Proposition \ref{desc} is an order -isomorphism by \cite[Theorem 32]{DPP4}. Let $T\in[R,S]$ be $\Pi$-irreducible.  Then, $\psi(T)=T(X)=TR(X)$ is $\Pi$-irreducible, so that $R\subset S$ is Boolean by Proposition \ref{desc}.
\end{proof} 

\begin{proposition} \label{desc 1} Let $R\subset S$  be a ring extension, $f: R \to R'$ a flat ring epimorphism  and $S':=  R'\otimes_RS$. 
  If  $ R \subset S$  is a distributive extension (resp. a  FIP Boolean extension),  then so is   $R'\subset S'$.
  \end{proposition}
  
  \begin{proof}  
The proof is a consequence of the following facts.  Let $f: R\to R'$ be a flat epimorphism and  $Q \in \mathrm{Spec}(R')$, lying over $P$ in $R$, then $R_P \to R'_Q$ is an isomorphism. Moreover,  we have 
  $(R'\otimes_RS)_Q \cong  R'_Q\otimes_{R_P} S_P$, so that $R_P\to S_P$ identifies to $R'_Q \to  (R'\otimes_RS)_Q=S'_Q $. 
  
  Assume that $ R \subset S$  is  distributive (resp.: FIP Boolean). Then, so is $R_P\to S_P$ for each $P \in \mathrm{Spec}(R)$ by Proposition \ref{1.014} (resp. Proposition \ref{4.02} and \cite[Proposition 3.7]{DPP2}).  Let $Q \in \mathrm{Spec}(R')$ and $P:=f^{-1}(Q)\in \mathrm{Spec}(R)$. Since $R_P\to S_P$ identifies to $R'_Q \to  S'_Q $, we get that $ R'_Q \subset S'_Q$  is  distributive (resp.; FIP Boolean)  for each $Q \in \mathrm{Spec}(R')$. It follows that $ R' \subset S'$  is distributive (resp.: FIP Boolean)  by the same references. Indeed, in the FIP Boolean case, since $R\subset S$ has FIP, so has $R'\subset S'$, because $\mathrm{Spec}(R')\to\mathrm{Spec}(R)$ is injective.
\end{proof} 

\begin{proposition}\label{4.0202}  Let $R\subset S$ be an FCP distributive extension. The following statements are equivalent:
\begin{enumerate}
\item $R\subset S$ is Boolean;

\item Any $\Pi$-irreducible element is  an atom;

\item  Any $\cap$-irreducible element is  a co-atom.
\end{enumerate}
\end{proposition}

\begin{proof}   (1) $\Rightarrow$ (2) and (3). Assume that $R\subset S$ is  Boolean. By  \cite[Theorem 5.2]{R}, the $\Pi$-irreducible elements of $[R,S]$ are the atoms of 
$[R,S]$. Using  complements, we deduce that the $\cap$-irreducible elements of $[R,S]$ are the co-atoms of 
$[R,S]$.

(2) $\Rightarrow$ (1) by the equivalences given in \cite[page 292]{St} since $R\subset S$ has FIP  by  Lemma \ref{5.7}. 

 (3) $\Rightarrow$ (2)  
 It is enough to exchange products  and intersections, and atoms and co-atoms. 
\end{proof}

\begin{theorem}\label{4.0203}  Let $R\subset S$ be an FCP  extension. The following statements are equivalent, each of them implying that $R\subset S$ has FIP:
\begin{enumerate}
\item $R\subset S$ is Boolean;

\item    $R\subset S$ is factorial;

\item   $R\subset S$ is co-factorial.
\end{enumerate}
\end{theorem}

\begin{proof} (1) $\Rightarrow$ (2) 
 by Theorem \ref{4.0} and (1) $\Rightarrow$ (3) because, 
 by Proposition \ref{5.5}, any $T\in [R,S]$ is a finite   intersection of  $\cap$-irreducible elements, and then a finite  intersection of  co-atoms of $[R,S]$ by Proposition \ref{4.0202}. 

(2)  $\Rightarrow$ (1). By Theorem \ref{5.8}, $R\subset S$ is an atomistic distributive   FIP extension. 
 Then, use the equivalences given in \cite[page 292]{St}.

(3) $\Rightarrow$ (2) 
 It is enough to exchange products  and intersections, and atoms and co-atoms. 
\end{proof}

\begin{proposition}\label{4.021} \cite[Theorems 107 and 158]{Do} Let $R\subset S$ be a Boolean extension. Then, $U\subset T$ is  Boolean  for any $U,T\in[R,S]$ such that $U\subseteq T$ and  the complement of  $V\in[U,T]$ in $[U,T]$ is $U(T\cap V^{\circ})$.
\end{proposition}

We can now generalize Ayache's result \cite[Theorem 7]{A2} in case of an arbitrary ring extension. 

\begin{proposition}\label{4.022} When  $R\subset S$    has a maximal chain of length $n$ from $R$ to $S$  such that  $|\mathrm{Supp}(S/R)|=|\mathrm{MSupp}(S/R)|=n$,   then $R\subset S$ is  FIP Boolean, any maximal chain  of $[R,S]$  has length $n$ and $|[R,S]|=~2^n$.

\end{proposition}

\begin{proof} Let $R_0:=R\subset\ldots\subset R_i\subset\ldots\subset R_n:=S$ be a maximal chain of $R$-subalgebras of length $n$. For each $i\in\mathbb N_n$, set $M_{i-1}:=\mathcal C(R_{i-1},R_i)\cap R$. Then, $\mathrm{Supp}(S/R)=\{M_i\}_{i= 0}^{n-1}\subseteq\mathrm{Max}(R)$ in view of Lemma \ref{1.9}. It follows that $M_i\neq M_j$ for each $i\neq j$, so that $R_{M_i}=(R_i)_{M_i}\subset(R_{i+1})_{M_i}=S_{M_ i}$ is minimal (and then has FIP), so that $R\subset S$ has FIP by \cite[Proposition 3.7]{DPP2}. Now, $R_{M_i}\subseteq S_{M_i}$ is Boolean (see Example \ref{4.01}(1)), and so is $R\subseteq S$ by Proposition \ref{4.02}. The last results follow from Theorem \ref{4.0}.
\end{proof}

Let $R\subset S$ be an FCP Boolean extension and let $R_0:=R\subset\ldots\subset R_i\subset\ldots\subset R_n:=S$ be a maximal chain of $R$-subalgebras of $S$.  By Proposition \ref{4.021}, $R\subset\ R_{n-1}$   is an Boolean  FCP extension and  $R_{n-1}\subset S$ is minimal. Next theorem gives a kind of converse which allows us to check by induction that an FIP extension is Boolean.
 
\begin{theorem}\label{4.03} An FIP extension $R\subset S$, which is not minimal,   is Boolean if and only if there exist $U,T\in]R,S[$ such that the conditions $\mathrm{(1), (2), (3)}$ and $\mathrm{(4)}$ hold, if and only if there exist $U,T\in]R,S[$ such that the conditions $\mathrm{(1), (4)}$ and $\mathrm{(5)}$ hold:
\begin{enumerate}
\item $[R,S]=[R,T]\cup[U,S]$.

\item $[U,S]=\{UL\mid L\in[R,T]\}$.

\item $L\subset UL$ is a minimal extension for each $L\in[R,T]$.

\item $[R,T]$ is a Boolean lattice.

\item The map $\varphi:[R,T]\to[U,S]$ defined by $\varphi(L)=UL$, for $L\in[R,T]$, is bijective.
\end{enumerate}

Moreover, if these conditions hold,  $U$ is an atom, $T= U^\circ$ is a co-atom.   In fact, these conditions hold for any atom $U'$ and its complement $T'$, and $[R,T']\cap[U',S]=\emptyset$. 
\end{theorem}

\begin{proof}   Assume  that $R\subset S$  has FIP 
 and is not minimal. Set $\mathcal A:=\{A_1,\ldots,A_n\}$.  
 We will prove the Theorem in four steps: 

(a) $R\subset S$ is  Boolean $\Rightarrow$ (1)+(2)+(3)+(4).

(b) (1)+(2)+(3)+(4)  $\Rightarrow$ (1)+(4)+(5).

(c) (1)+(4)+(5) $\Rightarrow$ (1)+(2)+(3)+(4).

(d) (1)+(2)+(3)+(4)+(5)  $\Rightarrow\ R\subset S$ is  Boolean.
 
 (a) Assume that $R\subset S$ is  Boolean. 
 Set $U:=A_1\in[R,S]$  and $T:=U^{\circ}\in[R,S]$,  the complement of $U$, so that $U\cap T=R$ and $UT=S$. Moreover, $T=\prod_{i=2}^nA_i$ \cite[Theorems 3.43 and 5.1]{R}. 

 (1) Let $L\in [R,S]$ and assume that $L\not\in[U,S]$. We have $U\subseteq L^{\circ}$,  the complement of $L$  \cite[Theorem 5.1]{R} and $U\subseteq L^{\circ}$ implies $L\subseteq T$ by \cite[Theorem 5.1]{R}, so that $L\in[R,T]$. Hence  (1)  is proved. Moreover, $[R,T]\cap[U,S]=\emptyset$ because  $U\subseteq L\subseteq T$, for some $L\in [R,T]\cap[U,S]$ leads to the contradiction $U\cap T=U=R$. 

(2) Each element of $[R,S]$ is a product of some $A_i$s by Theorem \ref{4.0203}. Let $L':=\prod_{i\in I}A_i\in[U,S]$ for some $I\subseteq\mathbb{N}_n$. Then, $1\in I$ because $U\subseteq L'$ (if not, $U\cap L'=R=U$ is absurd). In particular, $L'=UL$, where $L:=\prod_{i\in I\setminus\{1\} }A_i\subseteq T$ and (2) follows since $UL\in[U,S]$.  

  (3) Let $L:=\prod_{i\in I}A_i\in[R,T]$, so that $1\not\in I$ by (1). Then, $UL=\prod_{i\in I\cup\{1\}}A_i$. Let $L'\in[L,UL]$. There exists some $J\subseteq\mathbb{N}_n$ such that $L':=\prod_{i\in J}A_i$. By the uniqueness of this writing,  $I\subseteq J\subseteq I\cup\{1\}$, so that we have either $I=J$ or $J=I\cup\{1\}$, giving either $L'=L$ or $L'=UL$. It follows that $L\subset UL$ is  minimal  and (3) holds. 

 (4) By Proposition \ref{4.021} $R\subset T$ is Boolean. Remark that (1), (2), (3) and (4) hold for any atom $U'$  with complement $T'$.
  
In fact, we have the following diagram:
$\begin{matrix}
  R               & \rightarrow  &       L            & \rightarrow &       T            \\
\downarrow &          {}       & \downarrow  &        {}         & \downarrow \\
 U               & \rightarrow  &         UL        & \rightarrow &       S
\end{matrix}$
  
 (b) 
Assume that (1)+(2)+(3)+(4) holds. Then $\varphi$ is surjective by (2). Let $L,L'\in[R,T]$ be such that $\varphi(L)=\varphi(L')$, that is $UL=UL'$. Since $L,L'\subseteq T$, we get that $LL'\in[R,T]$. Moreover, $U(LL')=(UL)(UL')=UL$. Then, $L,L'\subseteq LL' \subseteq UL=UL'$ gives by (3) either $LL'\neq UL=UL'$, so that $LL'=L=L'$ or $LL'=UL\in[U,S]$. 
But, in this last case, $UL=LL'\subset ULL'=UL$ is minimal, a contradiction, because $LL'\in[R,T]$. 
 Then, 
 $\varphi$ is bijective and (5) holds.

 (c) 
Assume that (1)+(4)+(5) holds. Then (2) holds by (5). Let $L\in[R,T]$. Since $UL\in[U,S]$, we get $L\neq UL$. 
 Deny, so that $U\subseteq L\subseteq T$ yields $S=UT=T$, a contradiction. 
Assume that there exists some $L'\in]L,UL[$. It follows that $UL\subset UL'\subset U(UL)=UL$, a contradiction. Then, 
 $L\subset UL$ is minimal, giving (3).

 (d) Assume 
that 
 (1)+(2)+(3)+(4)+(5)  
holds for some $U,T\in[R,S]$ and  that $[R,T]\cap[U,S]\neq\emptyset$. If  $L\in[R,T]\cap[U,S]$,  then $L= UL$ is a contradiction with (3). So, $[R,T]\cap[U,S]=\emptyset$. We are going to prove that $[R,S]$ is a complemented distributive lattice. We first show that $T$ is the complement of $U$. We get that $R\subset U$ is minimal by (3), so that $R=U\cap T$, because $U\cap T=U$ would imply $U\subseteq T$, a contradiction. Indeed, $UT=S$ by (2) because the map $[R,T]\to [U,S]$ defined by $L\mapsto UL$  is increasing. Then, $T$ is a complement of $U$. 

 Now, conditions (1), (2)  
  and (5) show that $[R,S]$ is a {\it decomposable} lattice, that is for any $L\in[R,S]$, there is a unique pair $(L_1,L_2)\in[R,T]\times [R,U]$ such that $L=L_1L_2$ (see \cite[p.57]{Wh}). Set $L_1:=L$ and $L_2:=R$ when $L\in[R,T]$. Set  $L_2:=U$ and take $L_1\in[R,T]$ such that $L=\varphi(L_1)$ when $L\in [U,S]$. The uniqueness of such $(L_1,L_2)$ is obvious in each case. Then, $[R,S]$ is isomorphic as a lattice to $[R,T]\times [R,U]$. In particular, since $[R,T]$ and $[R,U]$ are each Boolean, and then distributive, so is $[R,S]$ by \cite[Proposition 3.5.1]{Wh}.

Let $L\in[R,S]=[R,T]\cup[U,S]$. Assume first that $L\in[R,T]$ with complement  $L''$  in $[R,T]$ because $R\subseteq T$ is Boolean, so that $L\cap L''=R$ and $LL''=T$. Setting $L':=L''U$, we get  $LL'=LL''U=TU=S$. 
 Moreover $L\cap L'=L\cap L''U=(L\cap L'')(L\cap U)=R$. Then, $L'$ is a complement of $L$ in $[R,S]$. 

Now, assume that $L\in[U,S]$. Set $L''=L\cap T\in[R,T]$, and let $L'\in[R,T]$ be its  complement  in $[R,T]$. Then, $L\cap L'=L\cap L'\cap T=L''\cap L'=R$. Moreover, $LL '=LUL'\supseteq UL''L'=UT=S$ gives that $LL'=S$. Then, $L$ has a complement in $[R,S]$. 

So, any element of $[R,S]$ has a complement in $[R,S]$,  which is unique by distributivity. 
To conclude, $R \subset S$ is  Boolean. 
\end{proof}

We are now in position to generalize and improve Ayache's result \cite[Theorem 38]{A1} for an  arbitrary FIP extension, using a completely different method. We will need  next results.

 \begin{proposition}\label{4.0201} An FCP extension $R\subset S$, such that $\mathrm{Supp}(T/R)\cap\mathrm{Supp}(S/T)=\emptyset$ for all $T\in[R,S]$,  is a Boolean extension. 
\end{proposition}

\begin{proof} By Theorem \ref{6.10}, $R_M\subset S_M$ is  minimal, whence Boolean for $M\in\mathrm{MSupp}(S/R)$, so   that $R\subset S$ is  Boolean   by Proposition \ref{4.02}.
\end{proof} 

\begin{proposition}\label{4.8} Let $R\subset S$ be an FCP extension and let $T\in [R, S]$ such that $\mathrm{MSupp}(S/T)\cap\mathrm{MSupp}(T/R)=\emptyset$. Then $R\subset S$ is  Boolean  if and only if $R \subset T$ and $T\subset S$ are  Boolean. 
\end{proposition}
\begin{proof}  Assume first that $R\subseteq T$ and $T\subseteq S$ are Boolean. Let $M\in\mathrm{MSupp}(S/R)$. Then either $M\in\mathrm{MSupp}(S/T)\ (*)$ or $M\in\mathrm{MSupp}(T/R)$ $\ (**)$. In case $(*)$, $M\not\in\mathrm{MSupp}(T/R)$, so that $T_M=R_M$. Hence, $[R_M,S_M]=[T_M,S_M]$ is Boolean by Proposition \ref{4.02}. In case $(**)$, $M\not\in\mathrm{MSupp}(S/T)$, so that $T_M=S_M$. Then, $[R_M,S_M]=[R_M,T_M]$ is Boolean by Proposition \ref{4.02}. Therefore, $[R_M,S_M]$ is Boolean  for each $M\in\mathrm{MSupp}(S/R)$, and Proposition \ref{4.02} gives that $R\subset S$ is  Boolean.

 The converse is given by Proposition \ref{4.021}.  
 \end{proof} 
 
 \subsection { Characterization of Boolean extensions}  

\begin{theorem}\label{4.9} An FIP extension $R\subset S$ is  Boolean if and only if $\mathrm{(1)}$ and ${(2)}$ hold, in which case $\widetilde R$ is the complement of $\overline R$:
\begin{enumerate}
\item $\mathrm{Supp}(\overline R/R)\cap \mathrm{Supp}(S/\overline R)=\emptyset$.

\item $[R,\overline R]$ and $[\overline R,S]$ are Boolean lattices.
\end{enumerate}
\end{theorem}

\begin{proof}  If $[R,S]$ is Boolean, set $T:=(\overline R)^{\circ}$. Then $\overline R\cap T=R$ and $\overline RT=S$, so that $\mathrm{Supp}(\overline R/R)\cap \mathrm{Supp}(S/\overline R)=\emptyset$ \cite[Proposition 3.6]{Pic 3}, and (1) holds. The same proposition shows that $\widetilde R=(\overline R)^{\circ}$.
Now (2)  results from Proposition \ref{4.8}.

Conversely, assume that (1) and (2) hold. Then, Proposition \ref{4.8} implies that $[R,S]$ is  Boolean.
\end{proof}

We get a generalization of Ayache's result \cite[Theorem 38]{A1} as a corollary thanks to the following lemma.

\begin{lemma}\label{4.91} Let $R\subset S$ be an FCP extension. Set $X:=\{(T',T'')\in [R,\overline R]\times[\overline R,S]\mid\mathrm{Supp}_{T'}(\overline R/T')\cap\mathrm{Supp}_{T'}(T''/\overline R)=\emptyset\}$. There is a bijection $\varphi:[R,S]\to X$ defined by $\varphi(T):=(T\cap\overline R,\overline RT)$ for each $T\in[R,S]$. In particular, if $R\subset S$ has  FIP, then $|[R,S]|\leq|[R,\overline R]||[\overline R,S]|$. 
\end{lemma}

\begin{proof} Let $(T',T'')\in[R,\overline R]\times[\overline R,S]$. Then, $\overline R$ is also the integral closure $\overline{T'}$ of $T'$ in $T ''$ (and in $S$). 

Let $T\in[R,S]$. Set $T':=T\cap\overline R$ and $T'':=\overline RT$. Then $(T',T'')\in [R,\overline R]\times[\overline R,S]$. If $T'=T''$, then $T'=T''=\overline R$ implies $T=\overline R$ and $\mathrm{Supp}_{T'}(\overline R/T')=\mathrm{Supp}_ {T'}(T''/\overline R)=\emptyset$. Now assume that $T'\neq T''$, then $\mathrm{Supp}_{T'}(\overline R/T')\cap\mathrm{Supp}_{T'}(T''/\overline R)=\emptyset$ \cite[Proposition 3.6]{Pic 3}. Hence we can define $\varphi:[R,S]\to X
$ by $\varphi(T):=(T\cap\overline R,\overline RT)$ for each $T\in[R,S]$.

Now, let $T_1,T_2\in[R,S]$ be such that $\varphi(T_1)=\varphi(T_2)=(T',T'')$. Assume $T'\neq T''$. Another use of \cite[Proposition 3.6]{Pic 3} gives that $T_1=T_2$. If $T'=T''$, then, $T'=T''=\overline R$, so that $T_1=T_2=\overline R$. It follows that $\varphi$ is injective. The same reference gives that $\varphi$ is bijective. 
\end{proof}

\begin{corollary}\label{4.92} The extension  $R\subset S$ is  FIP Boolean if and only if the two following conditions hold:

 \begin{enumerate}
\item $R\subseteq\overline R$ and $\overline R\subseteq S$ are FIP Boolean extensions.
\item $| [R, S] |=|  [R,\overline R]  ||  [\overline R,S]     |$.
\end{enumerate} 
\end{corollary}
\begin{proof}  If $R\subset S$ is FIP Boolean, Proposition \ref{4.021} implies that so are $R\subseteq \overline R$ and $\overline R\subseteq S$, 
 giving (1). Since $\mathrm{Supp}(\overline R/R)\cap\mathrm{Supp}(S/\overline R)=\emptyset$ by Theorem \ref{4.9},  we infer  that $|[R,S]|=|[R,\overline R]||[\overline R,S]|$ \cite[Lemma 3.7]{Pic 3}.

Conversely, assume that conditions (1) and (2) (which are equivalent to those of Theorem \ref{4.9} because $R\subset S$ has FIP) hold. The map $\varphi$ defined in Lemma \ref{4.91} being injective, condition (2) shows that $\varphi([R,S])=[R,\overline R]\times[\overline R,S]$. Therefore, there is $T\in[R,S]$ such that $\varphi(T)=(R,S)$, inducing by the properties of $\varphi$ that $\mathrm{Supp}_R(\overline R/R)\cap\mathrm{Supp}_R(S/\overline R)=\emptyset$, (actually, the condition of Theorem \ref{4.9}(1), asserting that $R\subset S$ is  Boolean).
\end{proof}

\begin{corollary}\label{4.93} Let $R\subset S$ be an FIP Boolean extension, where $R$ is a local ring. Then, $R\subset S$ is either Pr\"ufer, or integral.
\end{corollary}
\begin{proof} Use Theorem \ref{4.9} and \cite[Proposition 4.16]{Pic 5}. 
\end{proof} 

By Corollary \ref{4.93} and Proposition \ref{4.02}, the characterization of Boolean FIP extension $R\subseteq S$ can be reduced to those that are either Pr\"ufer or  integral, with $R$  local. 

For the Pr\"ufer case, we recover and generalize  some Ayache's results on  extensions of integral domains \cite[Proposition 35]{A1}. 

\begin{proposition}\label{4.10} An integrally closed extension $R\subseteq S$  is Boolean  and FIP  if and only if $R\subseteq S$ is a Pr\"ufer extension and $|\mathrm{Supp}(S/R)|=|\mathrm{MSupp}(S/R)|<\infty$. If these conditions hold, $|[R,S]|=2^{|\mathrm{Supp}(S/R)|}$. 
\end{proposition}
\begin{proof}  
Assume first that 
$|\mathrm{Supp}(S/R)|=|\mathrm{MSupp}(S/R)|<\infty$ and that $R\subseteq S$ is Pr\"ufer. From \cite[Proposition 6.9]{DPP2}, we infer that $R\subseteq S$ has FIP. Moreover, we proved in \cite[Proposition 6.12]{DPP2}, that, under these conditions, $|[R_M,S_M]|= 2$ for each $M\in\mathrm{Supp}(S/R)$, so that $R_M\subset S_M$ is Boolean for each $M\in\mathrm{Supp}(S/R)$ by Example \ref{4.01}(1), and then $[R,S]$ is Boolean by Proposition \ref{4.02}. From \cite[Proposition 6.12]{DPP2}, we deduce that $|[R,S]|=2^{|\mathrm{Supp}(S/R)|}$.

Conversely, suppose that $R\subseteq S$ is Boolean and has FIP. Then, \cite[Proposition 6.9]{DPP2} implies that $R\subseteq S$ is Pr\"ufer and $\mathrm{Supp}(S/R)$ is finite. Since $R\subseteq S$ is Boolean, so is $R _M\subset S_M$ for each $M\in\mathrm{Supp}(S/R)$ by Proposition \ref{4.02}. But, $[R_M,S_M]$ is chained \cite[Theorem 6.10]{DPP2}. It follows that $R_M\subset S_M$ is minimal for each $M\in\mathrm{Supp}(S/R)$ by Example \ref{4.01} (1), so that $|\mathrm{Supp}_{R_M}(S_M/R_M)|=1$ for each $M\in\mathrm {Supp}(S/R)$ and then $\mathrm{Supp}(S/R)\subseteq\mathrm{MSupp}(S/R)$ completes the proof.
\end{proof}

\begin{corollary}\label{4.11} An integrally closed FIP extension $R\subset S$, with $\mathrm{MSupp}(S/R):=\{M_1,\ldots,M_n\}$, is  Boolean if and only if $R\subset S$ is locally minimal. In this case, 
$R\subset S$ is a $\mathcal B$-extension.
\end{corollary} 
\begin{proof} The two properties  are equivalent by the proof of the above proposition. Since $\mathrm{Supp}(S/R)=\mathrm{MSupp}(S/R)$ holds in this case, \cite [Theorem 3.6]{DPP2} shows that $\varphi$ is a bijection.
\end{proof}

 The following definitions are needed for the sequel.
  
  \begin{definition}\label{1.3} An integral extension $R\subseteq S$ is called {\it infra-integral} \cite{Pic 2} (resp$.;$ {\it subintegral} \cite{S}) if all its residual extensions $\kappa_R(P)\to\kappa_S(Q)$, (with $Q\in\mathrm {Spec}(S)$ and $P:=Q\cap R$) are isomorphisms (resp$.$;and the natural map $\mathrm {Spec}(S)\to\mathrm{Spec}(R)$ is bijective). An extension $R\subseteq S$ is called {\it t-closed} (cf. \cite{Pic 2}) if the relations $b\in S,\ r\in R,\ b^2-rb\in R,\ b^3-rb^2\in R$ imply $b\in R$. The $t$-{\it closure} ${}_S^tR$ of $R$ in $S$ is the smallest element $B\in [R,S]$  such that $B\subseteq S$ is t-closed and the greatest element  $B'\in [R,S]$ such that $R\subseteq B'$ is infra-integral. An extension $R\subseteq S$ is called {\it seminormal} (cf. \cite{S}) if the relations $b\in S,\ b^2\in R,\ b^3\in R$ imply $b\in R$. The {\it seminormalization} ${}_S^+R$ of $R$ in $S$ is the smallest element $B\in [R,S]$ such that $B\subseteq S$ is seminormal and the greatest  element $ B'\in[R,S]$ such that $R\subseteq B'$ is subintegral. 
  The extension $R\subseteq {}_S^+R\subseteq {}_S^tR\subseteq \overline R\subseteq S$ is called the {\it canonical decomposition} of $R\subseteq S$.
 \end{definition}

Three types of minimal integral extensions exist, characterized in the next theorem, (a consequence of  the fundamental lemma of Ferrand-Olivier), so that there are four types of minimal extensions.

\begin{theorem}\label{minimal} \cite [Theorems 2.2 and 2.3]{DPP2} Let $R\subset T$ be an extension and  $M:=(R: T)$. Then $R\subset T$ is minimal and finite if and only if $M\in\mathrm{Max}(R)$ and one of the following three conditions holds:

\noindent (a) {\bf inert case}: $M\in\mathrm{Max}(T)$ and $R/M\to T/M$ is a minimal field extension.

\noindent (b) {\bf decomposed case}: There exist $M_1,M_2\in\mathrm{Max}(T)$ such that $M= M _1\cap M_2$ and the natural maps $R/M\to T/M_1$ and $R/M\to T/M_2$ are both isomorphisms; or, equivalently, there exists $q\in T\setminus R$ such that $T=R[q],\  q^2-q\in M$, and $Mq\subseteq M$.

\noindent (c) {\bf ramified case}: There exists $M'\in\mathrm{Max}(T)$ such that ${M'}^2 \subseteq M\subset M',\  [T/M:R/M]=2$, and the natural map $R/M\to T/M'$ is an isomorphism; or, equivalently, there exists $q\in T\setminus R$ such that $T=R[q],\ q^2\in M$, and $Mq\subseteq M$.
\end{theorem}

It remains to solve \cite[Problem 45]{A1}: under which conditions  an integral extension $R\subset S$ is    Boolean and has FIP? The study is quite complicated. We are going to use the canonical decomposition of an integral ring extension
and Proposition \ref{4.02}  which allows us to only
  consider    local rings $R$.  

\begin{proposition}\label{4.12} Let $R\subset S$ be an integral FIP Boolean   extension where $R$ is  local. Then, $R\subset S$ is either infra-integral or  t-closed. 
\end{proposition}
\begin{proof} Let $T:={}_S^tR$ be the t-closure of the local ring $(R,M)$ in $S$, and let $T^{\circ}\in[R,S]$ be its complement.  Let $T'$ be the t-closure of $R$ in $T^{\circ}$, so that $R\subseteq T'$ is infra-integral, and then  $T' \subseteq T$. It follows that  $ T' \subseteq T^{\circ}\cap T=R$. Then $T' =R$ and $R\subseteq T^{\circ}$  is t-closed. In the same way, let $T''$ be the t-closure of $T^{\circ}$ in $S$, so that $T''\subseteq S$ is t-closed and then $T\subseteq T''$. Hence  $S=T^{\circ}T\subseteq T''$, so that $T'' = S$ and $T^{\circ}\subseteq S$ is infra-integral. 

Assume that $R\subset S$ is neither t-closed, nor infra-integral, so that $T,T^{\circ}\neq R,S$. Then, there are $R_1\in[R,T]$ and $R'_1\in[R,T^{\circ}]$ such that $R\subset R_1$ is minimal infra-integral, and $R\subset R'_1$ is minimal inert, with both the same crucial maximal ideal $M$. By \cite[Propositions 7.1 and 7.4]{DPPS}, there are two maximal chains from $R$ to $R_1R_1'$ with different lengths, and the same statement holds for $R\subset S$. This contradicts Condition (JH) of Proposition \ref{1.0} since $R\subset S$ is  distributive. Then, $R\subset S$ is either infra-integral, or t-closed.
\end{proof}

\begin{remark} Proposition \ref{4.12} is no longer true if $R$ is not local. Take a ring $R$ with two distinct maximal ideals $M_1$ and $M_2$, and two minimal extensions $R\subset T_1$ ramified and $R\subset T_2$ inert with $M_1:=\mathcal{C}(R,T_1)$ and $M_2:=\mathcal{C}(R,T_2)$.  Assume that $S:=T_1T_2$ exists, so that $R_{M_1}\subset S_{M_1}=(T_1)_{M_1}$ is minimal ramified, and then Boolean and $R_{M_2}\subset S_{M_2}=(T_2)_{M_2}$ is minimal inert, and then Boolean. It follows from Proposition \ref{4.02} that $R\subset S$ is Boolean although being neither infra-integral nor t-closed. To get such a situation, we may take $S:=\mathbb{Z}[i],\ T_2:=\mathbb{Z}+2S,\ T_1:=\mathbb{Z}+3S$ and $R:=T_1\cap T_2=\mathbb{Z}+6S$. It is well known that $2$ is ramified in $S$ and $3$ is inert in $S$. Then, $T_2\subset S$ is a minimal ramified extension with conductor $2S$ and $T_1\subset S$ is a minimal inert extension with conductor $3S$. Moreover, $2S$ and $ 3S$ are incomparable. 
 Setting $M_1:=2S\cap R=2T_1$ and  $M_2:=3S\cap R=3T_2$, \cite[Proposition 6.6(a)]{DPPS} shows that  $R\subset T_1$ is  minimal ramified and $R\subset T_2$ is  minimal inert with $M_1:=\mathcal{C}(R,T_1),\ M_2:=\mathcal{C}(R,T_2)$ and  $S=T_1T_2$.
\end{remark}
 
  We first  
  consider the infra-integral case for which  we need the next lemma.

\begin{lemma}\label{4.16}  A subintegral FIP extension (resp. seminormal and infra-integral) $R\subset S$, where $(R,M)$ is local, is  Boolean if and only if $R\subset S$ is  minimal ramified (resp. decomposed).
\end{lemma}

\begin{proof} One implication is Example \ref{4.01} (1). 

Conversely, assume that $[R,S]$ is a Boolean lattice. The atoms of $[R,S]$ are of the form $R_i:=R+Rx_i$, for $i\in I:=\mathbb{N}_n,\ n:=|\mathcal{A}|$, where $x_i\in S$ is such that $x_i ^2\in M$ (resp. $x_i^2-x_i\in M$) and $x_iM\subseteq M$, because $R\subset R_i$ is minimal ramified (resp.; decomposed), with $M=(R:R_i)$. Then, $S=\prod_{i\in I}R_i$ by Theorem \ref{4.0}. Let $M_i:=M+Rx_i$ be the (resp.; one) maximal ideal of $R_i$. Assume that $n>1$. Let $i,j\in I$ be such that $i\neq j$, so that $R=R_i\cap R_j$, with $x_ix_j\in R_iR_j$. But, $R_i\subset R_iR_j$ and $R_j\subset R_iR_j$ are minimal Theorem \ref{4.0}, so that $x_i\in M_i=(R_i:R_iR_j)$ and $x_j\in M_j=(R_j:R_iR_j)$. In the decomposed case, we may choose $x_i$ and $x_j$ in order that $M_i$ and $M_j$ are the needed conductors. It follows that $M_i$ and $M_j$ are ideals of $R_iR_j$, and so is $M_i\cap M_j$, which contains $M_iM_j$. But $M_i\cap M_j\subseteq R_i\cap R_j=R$. This implies that $x_ix_j\in M_i\cap M_j=M_i\cap R\cap M_j=M$, the maximal ideal of $R$. Then, $x_ix_j\in M$, giving $S=R+\sum_{i\in I}Rx_i$. Set $x:=x_i+x_j\not\in R$ and $R_x:=R+Rx\neq R_i,R_j$. We get that $x^2=x_i^2+x_j^2+2x_ix _j\in M$ (resp. $x^2=x_i+x_j+m=x+m$, where $m\in M$). Moreover, $xM\subseteq x_iM+x_j M\subseteq M$, so that $R\subset R_x$ is minimal by Theorem \ref{minimal}, and $R_x$ is an atom of $[R,S]$. But we have $R_x\subseteq R_iR_j$, so that $R_x=R_x\cap R_iR_j=(R_x\cap R_i)(R_x\cap R_j)=R$, a contradiction. Then, $n=1,\ S=R_1$ and $R\subset S$ is minimal.
\end{proof}

 It may be  asked if  extensions of  Boolean rings and Boolean extensions are linked. Next result shows that they are quite never linked.

\begin{proposition}\label{4.199x} Let $S$ be a Boolean ring and $R$ a subring of $S$ such that $R\subset S$ is a finite extension. Then $R\subset S$ is seminormal infra-integral and  $R\subset S$ is Boolean if and only if $R\subset S$ is locally minimal.
\end{proposition}

\begin{proof} Since $S$ is a Boolean ring, $R\subset S$ is  seminormal integral. Because a Boolean local ring is isomorphic to $\mathbb Z/2\mathbb Z$, the residual extensions of $R\subset S$ are isomorphisms, so that $R\subset S$ is infra-integral. Now, $R\subset S$ is finite implies that it has FCP by \cite[Theorem 4.2]{DPP2}. 

Assume that $R\subset S$ is Boolean. Then, $R\subset S$ has FIP by Theorem \ref{4.0} because FCP distributive. Moreover, an appeal to Lemma \ref{4.16} shows that  
 $R\subset S$ is locally minimal. 
 
Conversely, if $R\subset S$ is locally minimal, then $R\subset S$ is Boolean in view of Corollary \ref{5.9}. Indeed, $R\subset S$ has FIP  \cite[Proposition 3.7]{DPP2}. 
\end{proof} 

Next proposition gives a characterization of arbitrary infra-integral   Boolean  FIP extensions. 

\begin{proposition}\label{4.17} An infra-integral FIP extension $R\subset S$, where $(R,M)$ is  local, is Boolean if and only if there exist $x,y\in S$ such that $S=R[x,y]$, where $x^2,xy,y^2-y\in M$ and $xM,yM\subseteq M$.

If these conditions hold, then $R[x,y]=R[x+y]$. 
\end{proposition}

\begin{proof} If $R\subset S$ is subintegral, we choose $y=0$, if $R\subseteq S$ is seminormal, we choose $x=0$. In both cases, we use Lemma \ref{4.16} to get the equivalence. From now on, we assume that $R\subset S$ is neither subintegral, nor seminormal. 

Assume that $R\subset S$ is Boolean and set $T:={}_S^+R\neq R,S$. From Proposition \ref{4.021}, we deduce that $[R,T]$ and $[T,S]$ are Boolean . Then, $R\subset T$ is minimal ramified, which implies that $T$ is also local. It follows that $T\subset S$ is minimal decomposed by Lemma \ref{4.16}, so that $\ell[R,S]=2$,  because all maximal chains of $[R,S]$ have the same length (\cite[Lemma 5.4]{DPP2}). If $U:=T^\circ$, then, $U\neq R,S$, so that $R\subset U$ and $U\subset S$ are minimal. Since $U\cap T=R$, we get that $R\subset U$ is decomposed, because it cannot be ramified. For the same reason, $U\subset S$ is minimal ramified. Let $x,y\in S$ be such that $T=R[x]$ and $U=R[y]$, so that $S=R [x,y]$, where $x^2,y^2-y\in M$ and $xM,yM\subseteq M$. Set $M':=M+Rx$ which is the only maximal ideal of $T$, so that $M'=(T:S)$. Set $M'':=M+Ry\in\mathrm{Max} (U)$. We can assume that $M''=(U:S)$. Then, $xy\in M'\cap M''\subseteq T\cap U\cap M'=M$. 

Conversely, assume that there exist $x,y\in S$ such that $S=R[x,y]$, where $x^2,xy, y^2-y\in M $ and $xM,yM\subseteq M$, so that $M=(R:S)$. Set $T:=R[x]$ and $ U:=R[y]$.  Then, $R\subset T$ is  minimal ramified and  $M+Rx$  is the maximal ideal of $T$, whereas, $R\subset U$ is  minimal decomposed and $M+Ry\in\mathrm {Max} (U) $. Moreover, $T\cap U=R$ and $TU=S$. 
Since $(M+Rx)(M+Ry)\subseteq M$, it follows from \cite[Proposition 7.6]{DPPS} that $R\subset S$ is an FCP infra-integral extension of length 2, so that $T\subset S$ is minimal decomposed and $U\subset S$ is minimal ramified. 
Hence, $[R, S]=\{R,T,U,S\}$ by \cite[Theorem 6.1]{Pic 6} and  is   Boolean by Theorem~\ref{4.03}. 

Assume that the preceding conditions hold, so that $R\subset S$ is Boolean, whence simple by Theorem \ref{4.0}. Since $R[x+y]\neq R,T,U$, 
we have $S=R[x+y]$. 
\end{proof}

We can now sum up the previous results in order to get a characterization of Boolean FIP extensions.

\begin{theorem}\label{4.18} An FIP extension  $R\subset S$  is  Boolean  if and only if, for each $M\in\mathrm{MSupp}(S/R)$, one of the following conditions holds:

\begin{enumerate}
\item $R_M\subset S_M$ is a minimal extension.

\item There exist $U,T\in[R_M,S_M]$ such that $R_M\subset T$ is minimal ramified, $R_M\subset U$ is minimal decomposed and $[R_M,S_M]=\{R_M,T,U,S_M\}$.

\item $R_M\subset S_M$ is a Boolean t-closed extension 
 (equivalently, $R_M\subset S_M$ is   t-closed  and $\kappa_R(M)\subset \kappa_S(N)$ is a Boolean field extension, where $N$ is the only maximal ideal of $S$ lying above $M$). 
\end{enumerate}
\end{theorem}

\begin{proof} From Proposition \ref{4.02}, we deduce that $[R, S]$ is Boolean if and only if, for each $M\in\mathrm{MSupp}(S/R),\ [R_M,S_M]$ is Boolean. An appeal to Corollary \ref{4.93} shows that this statement is equivalent to the following: for each $M\in\mathrm{MSupp}(S/R)$, either $R_M\subset S_M$ is Boolean and integrally closed $(*)$, or $R_M\subset S_M$ is Boolean integral $(**)$. In case $(*)$ Corollary \ref{4.11}  gives that $R_ M\subset S_M$ is Boolean integrally closed if and only if $R_M\subset S_M$ is Pr\"ufer minimal. In case $(**)$, when $R_M\subset S_M$ is integral, Proposition \ref{4.12} says that $R_M\subset S_M$ is Boolean if and only if either $R_M\subset S_M$ is infra-integral Boolean (a), or $R_M\subset S_M$ is t-closed Boolean (b). To sum up, if $[R, S]$ is Boolean, for each $M\in\mathrm{MSupp}(S/R)$, either $(*)$ or $(**)$ (a) or $(**)$ (b) holds. Case $(*)$ implies (1). Case $(**)$ (a) implies either (1) or (2). Indeed, by Proposition \ref{4.17}, there exist $x,y\in S_M$ such that $S_M=R_M[x,y]$, where $x^2,xy,y^2-y\in MR_M$ and $xMR_M,yMR_M\subseteq MR_M$. Since $M\in\mathrm{MSupp}(S/R)$, we have $R_M\neq S_M$, so that either $x\not\in R_M$ or $y\not\in R_M$. If both $x,y\not\in R_M$, setting $T:=R_M[x]$ and $U:=R_M[y]$, we get (2). If only one of $x,y\not\in R_M$, we get (1). Case $(**)$ (b) is (3). Conversely, if (1) holds, then $R_M\subset S_M$ is Boolean by Example \ref{4.01}(1). If (2) holds, then $R_M\subset S_M$ is Boolean by Proposition \ref{4.17}, setting $T:=R_M[x]$ and $U:=R_M[y]$ for some $x\in T,\ y\in U$ such that $x^2,xy,y^2-y\in MR_M$ and $xMR_M,yMR_M\subseteq MR_M$. At last, (3) is $(**)$ (b). 

 We now  show that, for some $M\in\mathrm{MSupp}(S/R),\ R_M\subseteq S_M$ is Boolean t-closed is equivalent to  $R_M\subseteq  S_M$ is   t-closed  and $\kappa(M)\subset \kappa(N)$ is a Boolean field extension, where $N$ is the only maximal ideal of $S$ lying above $M$. Set $R':=R_M,\ M':=MR_M$ and $S':=S_M$. Since $R'\subset S'$ is   t-closed, $M'=(R':S')$ and $(S',M')$ is local, \cite[Lemma 3.17]{DPP3}.
 It follows that there is only one maximal ideal $N$  in $S$ lying over $M$, so that $S_M=S_N$  (\cite[Proposition 2, page 40]{Bki A1}) and $\kappa_S(N)=S'/M'$. Then  $\kappa_R(M)=R'/M'\subseteq S'/M'=\kappa_S(N)$ is an FIP field  extension. Now, $R'\subseteq S'$ is Boolean   if and only if $R'/M'\subseteq S'/M'$ is Boolean by Proposition \ref{4.02}. 
\end{proof}

It follows that the remaining t-closed case can be reduced to the case of  FIP field extensions.   The case of fields is the subject of the next section, because of its complexity.

\begin{proposition}\label{4.19} A ring extension $R\subseteq S$  has FIP and is Boolean if and only if $R(X)\subseteq S(X)$ has FIP and is Boolean.
\end{proposition}

\begin{proof} One part of the proof is Proposition \ref{4.1999}. So, assume that $R\subseteq S$ has FIP and is Boolean. Then,    $R(X)\subseteq S(X)$ has FIP if and only if 
 $R\subseteq{}_S^+R$ is arithmetic 
 \cite[Corollary 4.3]{Pic 4}. If this conditions holds,  the map $\psi:[R,S]\to [R(X),S(X)]$ defined by $T\mapsto T(X)$ is an order-isomorphism \cite[Theorem 32]{DPP4}. 
It follows that $R(X)\subseteq S(X)$ is Boolean since Boolean conditions are preserved  through $\psi$. 
 To complete the proof, we need only to show that 
 $R(X)\subseteq S(X)$ has FIP. 
 It follows from Proposition \ref{4.021} that $[R,{}_S^+R]$ is finite and Boolean. But, under these conditions, Theorem \ref{4.18} yields that either $R_M=({}_S^+R)_M$ or $R_M\subset({}_S^+R)_M$ is minimal, for each $M\in\mathrm{MSupp}(S/R)$. This means that $R\subseteq{}_S^+R$ is arithmetic, so  that $R(X)\subseteq S(X)$ has FIP.
\end{proof}

\section {Boolean FIP field extensions}

The characterization of a Boolean extension of fields is quite different from those obtained in Theorem \ref{4.03} and  needs  a special study. 
\subsection{ FIP non separable field extensions}

We will call in this paper {\it radicial} any purely inseparable field extension. We recall that a minimal field extension is either separable or radicial (\cite[p. 371]{Pic}). We will use the separable closure of a FIP algebraic field extension. 
 In this subsection, we  only consider  FIP field extensions. Indeed, a finite algebraic field extension is not necessarily FIP. For instance a radicial extension $k\subseteq  L$ has not FIP, when  $p:=\mathrm{c}(k)$, $L:=k[x,y]$, with $x^p,y^p\in k$ and $[L:k]=p^2$. 

\begin{lemma}\label{5.1} An FIP radicial field extension $k\subseteq K$  is chained.
\end{lemma}

\begin{proof} Since $k\subseteq K$ has FIP, there exists $\alpha\in K$ such that $K=k [\alpha]$ by the Primitive Element Theorem. Moreover, $\mathrm{c}(k)=p$ is a prime integer  because $k\subseteq K$ is  radicial. Then, the monic minimal polynomial of $\alpha$ over $k$ is of the form $f(X):=X^{p^n}-a=(X-\alpha)^{p^n}$, where $a:=\alpha^{p^n}$ for some positive integer $n$.

The map   $\varphi:\{0,\ldots,n\}\to [k,K]$  defined by $\varphi(m):=k[\alpha^{p^m}]$ is  strictly 
 decreasing. Let $L\in[k,K]$ and $g(X)$ be the monic minimal polynomial of $\alpha$ over $L$. Then, $g(X)$ divides $f(X)$ in $K[X]$ and is of the form $g(X)=(X-\alpha)^{p^m}$ for some $m\in\{0,\ldots,n\}$ because $L\subseteq K=L[\alpha]$ is  radicial  and then the degree of $g(X)$ is a power of $p$. It follows that $g(X)=X^{p^m}-\alpha^{p^m}=X^{p^m}-\beta$, where $\beta:=\alpha^{p^m}\in L$. By the proof of the Primitive Element Theorem,  $L$ is generated over $k$ by the coefficients of $g(x)$, so that $L=k[\beta]=k[\alpha^{p^m}]=\varphi(m)$. Then, $\varphi$ is a bijection and $[k,K]$ is chained.
\end{proof}

We here take the opportunity   to correct a miswriting  in the proof of \cite[Proposition 2.3]{Pic 4}. The sentence: ``It follows that there is only one maximal chain composing $K\subseteq L$, and it has length $n$'' has to be replaced with ``It follows that any maximal chain composing $K\subseteq L$ has length $n$''.

\begin{theorem}\label{4.199} An FIP field extension  $k\subset K$, with separable  closure $T$ and radicial closure $U$ such that   $U,T\not\in\{k,K\}$,  is  Boolean  if and only if the following conditions hold: 

\begin{enumerate}
\item $k\subset U$ and $T\subset K$ are minimal.

\item $[k,K]=[k,T]\cup[U,K]$.

\item $k\subset T$ and $k\subset U$ are linearly disjoint.

\item $[k,T]$ is a Boolean lattice.
\end{enumerate}

If these conditions hold, then $[k,T]\cap[U,K]=\emptyset$ and $U = T^\circ$.
\end{theorem}

\begin{proof} Since $k\subset U$ is radicial,   
  $\mathrm{c}(k)=p$ 
is  a prime integer. 

Assume that $k\subset K$ is Boolean. If $T^{\circ}$ is the complement of $T$ and $T'$ the separable closure of $k$ in $T^{\circ}$, then $T'\subseteq T^{\circ}\cap T=k$ entails that $k\subseteq T^{\circ}$ is radicial, so that $T^{\circ}\subseteq U$. But $K=T^{\circ}T\subseteq UT$ gives $UT=K$. Moreover, $k=U\cap T$ shows that $T^{\circ}=U$.

We claim that $k\subset U$ is minimal. Deny, and let $U_1\in]k,U[$ be such that $k\subset U_1$ is minimal. Since  $k\subset U$ is radicial, $[k,U]$ is a chain by Lemma \ref{5.1}. Let $U^{\circ}_1\in]T, K[$ be the complement of $U_1$ in $[k,K]$. Then, $K=U_1U^{\circ}_1\subseteq UU^{\circ}_1$ implies $K=UU^{\circ}_1$. Moreover, $k\subseteq U\cap U^{\circ}_1\subseteq U\Rightarrow U\cap U^{\circ}_1\in[k,U]$. If $k\neq U\cap U^{\circ}_1$, then $U_1\subseteq U\cap U^{\circ}_1\subseteq U^{\circ}_1\Rightarrow U_1U^{\circ}_1=U^{\circ}_1=K$ is absurd because $U_1\neq k$. Hence, $U\cap U^{\circ}_1=k$, so that $U^{\circ}_1$ is also the complement of $U$, which is absurd. It follows that $k\subset U$ is minimal.  
 Then (1), (2) and (4) hold by Theorem \ref{4.03}(3),(1),(4). We get that $[U:k]=[K:T]=p$ since $k\subset U$ and $T\subset K$ are minimal radicial. But, $K=TU\Rightarrow[TU:T]=[U:k]$, and then $k\subset T$ and $k\subset U$ are linearly disjoint  \cite[Proposition 5, A V.13]{Bki A}. 

Conversely, assume that (1), (2), (3) and (4) hold. Then, the  above conditions (2) and (4) on $T$ and $U$ coincide with  Theorem \ref{4.03}(1),(4), so that we can use the diagram  appearing in  its proof, where $L\in [k,T]$:
$$\begin{matrix}
   k               & \rightarrow  &       L            & \rightarrow &       T            \\
\downarrow &          {}        & \downarrow &        {}         & \downarrow \\
 U                & \rightarrow  &         UL       & \rightarrow &       K
\end{matrix}$$
Then, $L\subseteq LU$ and $L\subseteq T$ are linearly disjoint \cite[Proposition 8, A V.14]{Bki A}. In particular, $[LU:L]=[K:T]=p$ shows that $L\subset LU$ is minimal as in Theorem \ref{4.03}(3). Let $L\in[U,K]$, and set $ L':=L\cap T$, which is the separable closure of $k\subseteq L$. Then, $L'\subseteq L$ is radicial. But, $UL'\subseteq L$ is both separable (because so is $U\subseteq K$ by \cite[Proposition 16, page 45]{Bki A}) and radicial, because $L'\subseteq L$ is radicial, so that $L'U=L=U(L\cap T)$. Then, $[U,K]=\{UL'\mid L'\in[k,T]\}$ implies Theorem \ref{4.03}(2). Finally, $[k,K]$ is Boolean by Theorem ~\ref{4.03}. The missing statement $[k,T]\cap[U,K]=\emptyset$ hold by Theorem~\ref{4.03}.  
\end{proof}

Hence, we need only to consider either radicial or separable Boolean field extensions to have a complete answer. 
 We recall  \cite{GHe} that a finite field extension $k\subset L$ is said to be {\it exceptional} if $k$ is the radicial closure and $L$ is not the separable closure. Then, a finite  exceptional field extension $k\subset L$ is never Boolean. Deny. Using notation of Theorem \ref{4.199}, $U=k=T^{\circ}$ implies $T=L$, a contradiction.

\begin{proposition}\label{4.21x} An  FIP radicial field extension $k\subset K$ is Boolean if and only if $k\subset K$ is  minimal  if and only if  $\mathrm{c}(k)=[K:k]$.
\end{proposition}

\begin{proof} Use Example \ref {4.01} (1) since $[k,K]$ is a chain  by Lemma \ref {5.1} for the first equivalence, the second comes from \cite[Proposition 2.2]{Pic}. 
\end{proof}

A Galois extension $k\subset K$ is  minimal  if and only if  $[K:k]$ is a prime integer  \cite[Proposition 2.2]{Pic}. But this equivalence does not always hold for an arbitrary finite separable extension.
  Theorem  \ref {4.210}  characterizes minimal separable extensions, independently  of Galois Theory, contrary to  Philippe's  methods  \cite{Ph}.   She proved that  a separable extension $k\subset k(x)$ is minimal if and only if the Galois group of the minimal polynomial of $x$ is primitive \cite[Proposition 2.2(3)]{Pic}. See also Cox \cite[page 414]{Co} for the characterization and links between primitive groups and ``primitive'' separable polynomials, a non trivial theory.

\subsection{Finite separable field extensions}
Let $k\subset L$ be a finite separable field extension,
 (whence FIP). 
Then any $T\in[k,L[$ is an intersection of finitely many $\cap$-irreducible elements of $[k,L]$ by Proposition \ref {5.5}. We  give an  upper bound of $|[k,L]|$ and recall a result from Dobbs-Mullins.
 
\begin{proposition}\label{4.21} Let $k\subset L$ be a finite  separable field extension of degree $n$. 
  
(1) Then, $|[k,L]|\leq B_n$, where $B_n$ is the $n$th Bell number.
  
(2)  \cite[Theorem 2.7]{DM} If $k$ is an infinite field, then, $|[k,L]|\leq 2^{n-2}+1$.  
\end{proposition}

\begin{proof} (1) $k\subset L$ is \'etale since a finite separable field extension, and then, letting $\Omega$ be an algebraic closure of $k$, \cite[Proposition 2, page AV.29]{Bki A} shows that the $\Omega$-algebra $\Omega\otimes_kL$ is diagonalizable, and then isomorphic to $\Omega^n$, for some integer $n$. But $|[\Omega,\Omega^n]|=B_n$ by \cite[Proposition 4.15]{DPP3}, so that  $[k,L]\leq B_n$. 
\end{proof} 
 
Moreover, if $k \subset L$ is Boolean, any $\cap$-irreducible element is a co-atom by Proposition \ref{4.0202}. In fact, Theorem \ref{4.0203} says that $k \subset L$ is Boolean if and only if any $T\in[k,L[$ is an intersection in a unique way of finitely many co-atoms. Thanks to principal subfields introduced in \cite{HKN} and some of their properties we studied in \cite{Pic 6}, we are able to characterize co-atoms of a finite separable field extension, and then give a characterization of a finite separable Boolean field extension by using \cite{HKN}, 
 from 
van Hoeij, Kl\"uners and Novocin, that gives an algorithm to compute subextensions of a finite separable field extension. 
 We recall the notation of  \cite{Pic 6},
($k_u[X]$ is the set of monic polynomials of $k[X])$. 

From now on, our riding hypotheses for the subsection will be: $L:=k[x]$ is a finite separable (whence FIP) field extension of $k$  with degree $n$ and $f(X)\in k_u[X]$ is 
 {\it the} 
 minimal polynomial of $x$ over $k$. 
If $g(X)\in L_u[X]$ divides $f(X)$, we denote by $K_g$ the $k$-subalgebra of $L$ generated by the coefficients of $g$.  For any $K\in[k,L]$, we denote by $f_K(X)\in K_u[X]$ the minimal polynomial of $x$ over $K$. The proof of the Primitive Element Theorem shows that $K=K_{f_K}\ (*)$.
Of course, $f_K(X)$ divides $f (X)$ in $K[X]$ (and in $L[X]$). 
If $f(X):=(X-x)f_1(X)\cdots f_r(X)$ is the decomposition of $f(X)$ into irreducible factors of $L_u[X]$, we  set $\mathcal F:=\{f_1,\ldots,f_r\}$  because the $f_{\alpha}'$s are different by separability. 
 For each $\alpha\in\mathbb N_r$, we set $L_{\alpha}:=\{g(x)\in L\mid g(X)\in k[X],\  g(X)\equiv g(x)\ (f_{\alpha}(X))$ in $L[X]\}$. 
The $L'_{\alpha}$s are called the {\it principal subfields} of $k\subset L$. 
It may be that $L_{\alpha}=L_{\beta}$ for some $\alpha\neq \beta$ (see \cite[Example 5.17 (1)]{Pic 6}). To get rid of this situation, 
 we defined in \cite{Pic 6} $\Phi:\mathcal F\to [k,L]$ by $\Phi(f_{\alpha})= L_{\alpha}$. If $t:=|\Phi(\mathcal F)|$, we set $\Phi(\mathcal F):=\{E_1,\ldots,E_t\}:=\mathcal E$. We denote by $E_{\beta}$  the common value of these $L_{\alpha}$'s and by $\mathcal E:=\Phi(\mathcal F)$  the set of all distinct principal subfields of $k \subset L$.  
 We set  $m_{\beta}:=f_{E_{\beta}}$ for $\beta\in\mathbb N_t$.

For $K\in[k,L[$ , we set $I(K):=\{\alpha\in  \mathbb N_r \mid f_K(X)/ f_{\alpha}(X) \in L_u[X] \}$.
 We also set $J(K):=\{\beta\in  \mathbb N_t\mid \Phi(f_{\alpha})=E_{\beta}$ for all $\alpha\in I(K)\}$.  
 For each $\beta\in\mathbb N_t$, we set $\mathcal{F}_{\beta}:=\{f_{\alpha}\in \mathcal{F}\mid f_{\alpha}$ divides $m_{\beta}$ in $L[X]\}$.

\begin{theorem}\label{4.210} \cite[Theorem 5.5]{Pic 6} Let $K\in[k,L[$. Then,  $f_K(X)=(X-x)\prod_{\alpha\in I(K)}f_{\alpha}(X)$ and $K=\{g(x)\in L\mid g(X)\in k[X], g(X)\equiv g(x)\ (f_K(X))$ in $K[X]\}=\cap_{\beta\in J(K)}E_{\beta} =\cap_{\alpha\in I(K)}L_{\alpha}$. In particular, $|[k,L]|\leq 2^t$ 
 and $k\subset L$ is minimal if and only if $t=1$. 
\end{theorem} 

 \begin{proof}  The inequality comes from that any $K\in[k,L]$ is an intersection of some principal subfields and gives the equivalence.
\end{proof} 

The inequality in Theorem \ref{4.210} gives a better bound than the one of Proposition \ref{4.21} because $2^t\leq 2^{n-1}\leq B_n$ thanks to the induction formula $B_{n+1}=\sum_{i=0}^n\mathrm C_n^{i}B_i$. In case $k\subset L$ is not a Galois extension, we get a better bound than the 
  bound of
 \cite[Theorem 2.7]{DM}.

\begin{corollary}\label{Bell} Let $k\subset L $ be a finite  separable field extension which is not Galois and of degree $n$. Then, $|[k,L]|\leq  2^{n-2}$. 
\end{corollary}

\begin{proof} Let $f(X)$ be the  minimal polynomial of the extension $k\subset L$ and set $f(X) :=(X-x)f_1(X)\cdots f_r(X)$  the decomposition of $f(X)$ into irreducible factors of $L_u[X]$. Since $k\subset L$ is not  Galois, at least one $f_i$ is not of degree 1, so that $n=1+\sum_{i=1}^r\deg(f_i)\geq 1+2+(r-1)=r+2$. It follows that  $t\leq r\leq n-2$ implies $|[k,L]|\leq  2^t\leq 2^{n-2}$.
\end{proof}

 In the following, we write $K_{\alpha}:=K_{g_{\alpha}}$, where $g_{\alpha}(X):=(X-x)f_{\alpha}(X)$.

\begin{lemma}\label{4.22}  
For $K,K'\in[k,L]$, $K\subseteq K'\Leftrightarrow f_{K'}(X) | f_K(X)$ in $L[X]$. 
\end{lemma}

\begin{proof} 
Assume that $K\subseteq K'$. Then, $f_K(X)\in K'[X]$  satisfies $f_K(x)=0$, so that $f_{K'}(X)$ divides $f_K(X)$ in $K'[X]$, and also in  $L[X]$. 

Conversely, assume that $f_{K'}$ divides $f_K$ in $L[X]$. Since Theorem \ref{4.210} implies $K=\cap_{\alpha\in I(K)}L_{\alpha},\ K'=\cap_{\alpha\in I(K')}L_{\alpha}$ and $f_{K'}$ divides $f_K$ in $L [X]$, 
any $f_{\alpha}$ which divides $f_{K'}$ divides $f_K$, so that $I(K')\subseteq I(K)$ and then 
$K=\cap_{\alpha\in I(K)}L_{\alpha}\subseteq\cap_{\alpha\in I(K)}L_{\alpha}=K'$. 
\end{proof}

\begin{proposition}\label{4.2501} If $K,K'\in[k,L]$, then {\rm lcm}$(f_K,f_{K'})$ divides $f_{K\cap K'}$ and $f_{KK'}$ divides {\rm gcd}$(f_K,f_{K'})$ in $L[X]$.
\end{proposition}

\begin{proof} Use Lemma \ref{4.22} applied to $K\cap K'\subseteq K,K'\subseteq KK'$. 
\end{proof}

 We set $\mathcal D:=\{f_K\mid K\in[k,L]\}$.
 Then, $(\mathcal D,\leq)$ is a poset for the order $\leq$ defined as follows: if $f_K,f_{K'}\in\mathcal D$, then $f_K\leq f_{K'}$ if and only if $f_K|f_{K'}$ in $L[X]$, which is equivalent to $ K'\subseteq K$ by Lemma \ref{4.22}. In particular, $\sup$ and $\inf$ are respectively {\rm lcm} and {\rm gcd} in $\mathcal D$.  

\begin{corollary}\label{4.251} The map  $\varphi :[k,L]\to\mathcal D$ defined by $K\mapsto f_K$ is a reversing order bijection such that $\sup(f_K,f_{K'})=f_{K\cap K'}$ and $f_{KK'}=\inf(f_K,f_{K'})$ for $K,K'\in[k,L]$.
\end{corollary}

\begin{proof} $\varphi$ is obviously surjective and is injective since $K=K_{f_K}$ by $(*)$. It is reversing order by Lemma \ref{4.22}. Let $K,K'\in[k,L]$, we deduce from Proposition \ref{4.2501} that $f_{KK'}\leq f_K,f_{K'}\leq f_{K\cap K'}$. Let $K_1\in[k,L]$ be such that $f_{K_1}\leq f_K,f_{K'}$. It follows that $K,K'\subseteq K_1$ so that $KK'\subseteq K_1$, whence  $f_{K_1}\leq f_{KK'}$ and then, $f_{KK'}=\inf(f_K,f_{K'})$. A similar proof shows that $\sup(f_K,f_{K'})=f_{K\cap K'}$.
\end{proof}

We denote by  $\mathcal {CA}:=\{K\in[k,L]\mid K\subset L\ {\rm minimal}\}$ the set of co-atoms of $[k,L]$. 

\begin{proposition}\label{4.26}  Assume that $k\subset L$ is not minimal and let $K\in]k,L[$. If $K\in{\mathcal {CA}}$, there is some $\beta\in\mathbb N_t$ such that $K=E_{\beta}$. 
Moreover, for any $\beta\in\mathbb N_t$, the following conditions are equivalent:

\begin{enumerate}
\item $E_{\beta}\in{\mathcal {CA}}$.

\item  $\mathcal{F}_{\beta}$ is a minimal element in the set $\{\mathcal{F}_{\gamma}\mid \gamma\in\mathbb N_t\}$.

\item $\varphi(E_{\beta})$ is a minimal element in $\mathcal D\setminus\{X-x\}$.
\end{enumerate}
\end{proposition}

\begin{proof} By \cite[Lemma 5.10]{Pic 6}, $K=E_{\beta}$ for some $\beta\in\mathbb N_t$ since $K$ is $\cap$-irreducible. Moreover, $m_{\beta}(X)=(X-x)\prod_{f_{\alpha}\in\mathcal{F}_{\beta}}f_{\alpha}(X)$ by the definition of $\mathcal{F}_{\beta}$. 

(1) $\Rightarrow$ (2) Assume that $E_{\beta}\subset L$ is minimal. We claim that $\mathcal{F}_{\beta}$ is minimal in the poset $\{\mathcal{F}_{\gamma}\mid \gamma\in\mathbb N_t\}$. Deny, then there is some $\beta '\in\mathbb N_t$ such that $\mathcal{F}_{\beta '}\subset\mathcal{F}_{\beta}$, so that $m_{\beta '}$ divides strictly $m_{\beta}$. We get $E_{\beta}\subset E_{\beta '}\subset L$, contradicting  $E_{\beta}\subset L$  minimal and  (2)  holds.

(2) $\Rightarrow$ (3) Let $E_{\beta}$ for some $\beta\in\mathbb N_t$ satisfying (2) and assume that $\varphi(E_{\beta})$ is not minimal in $\mathcal D\setminus\{X-x\}$. Then, there is some $K\in[k,L]$ such that $f_K$ divides strictly $m_{\beta}$. It follows that $E_{\beta}\subset K$ by Lemma \ref{4.22}. But $K$ is an intersection of some $E_{\gamma}$'s by Theorem \ref{4.210}. In particular, we have $E_{\beta}\subset K\subseteq E_{\gamma}$ which implies that $m_{\gamma}$ divides strictly $m_{\beta}$, so that $\mathcal{F}_{\gamma}\subset\mathcal{F}_{\beta}$, a contradiction with (2).

(3) $\Rightarrow$ (1) Let $E_{\beta}$ for some $\beta\in\mathbb N_t$ satisfying (3). Assume that $E_{\beta}\subset L$ is not minimal, so that there exists some $K\in[k,L]$ such that $E_{\beta}\subset K\subset L$. Using again Theorem \ref{4.210}, we exhibit some $E_{\gamma}$ such that $K\subseteq E_{\gamma}\subset L$, giving $E_{\beta}\subset K\subseteq E_{\gamma}$. Then, $m_{\gamma}$ divides strictly $m_{\beta}$, contradicting (3) and then, $E_{\beta}\in\mathcal {CA}$.
\end{proof}

 In case $k\subset L$ is Galois, we can give a characterization of $\mathcal{CA}$ from the Galois group of the extension.

\begin{proposition}\label{4.269}  Let $k\subset L$ be a finite Galois extension with $n:=[L:k]=\prod _{i\in \mathbb N_m}p_i^{e_i}$ and Galois group $G$. 
\begin{enumerate}
\item Let $K\in[k,L[$. Then $K\in \mathcal{CA}\Leftrightarrow$ there exists some $i\in \mathbb N_m$ such that $[L:K]=p_i\Leftrightarrow$ there exists some subgroup $H$ of $G$ 
 of order $p_i$
 such that $K$ is the fixed field of $H$ in $L$.

\item $|\mathcal{CA}|\geq m$.
\end{enumerate}
\end{proposition}
\begin{proof}  (1) An appeal to the Fundamental Theorem of Galois Theory shows that $K\in \mathcal{CA}\Leftrightarrow$ the group $H$ of $K$-automorphisms of $L$ has no proper subgroup $\Leftrightarrow |H|=p_i$ for some $i\in \mathbb N_m$ since $|G|=n$ $Ê\Leftrightarrow$ there exists some $i\in \mathbb N_m$ such that $[L:K]=|H|=p_i$ because $K$ is the fixed field of $H$. 

(2) For each $i\in \mathbb N_m$, there exists a subgroup of $G$ of order $p_i$ and therefore   an element of $\mathcal{CA}$ by (1), which yields $|\mathcal{CA}|\geq m$.
\end{proof} 

Since each element of $\mathcal{CA}$ is some $E_{\beta}$, we can reorder them so that $\mathcal{CA}=\{E_1,\dots,E_s\}$ with $s\leq t$.

If $k\subset L$ is a finite separable  field extension, Theorem \ref{4.210} tells us that any $K\in[k,L]$ is an intersection of  some $E_{\beta}$s, that we can suppose $\cap$-irreducible. In order to have a Boolean extension, any irreducible $E_{\beta}$ must belong to $\mathcal{CA}$.

 \begin{remark}\label{4.329} An $\cap$-irreducible element is not necessarily a co-atom. It is enough to take a Galois cyclic extension $k\subset L$ such that $[L:k]=p^n,\ n\geq 3$, where $p$ is a prime integer. Then the Galois group of the extension is a chain, and so is $[k,L]$, with $\ell[k,L]=n$. Any $K\in[k,L]$ such that $\ell[k,K]\in\{1,\ldots,n-2\}$ is  $\cap$-irreducible but not a co-atom.
\end{remark} 

 We have seen in Lemma \ref{4.22} that for each $K\in[k,L]$, there exists $g(X)=(X-x)g'(X)$, where $g'(X)\in L[X]$ is a product of some of the $f_{\alpha}(X)$, and satisfying $g=f_K$. Let $g(X)=(X-x)g'(X)$, where $g'(X)\in L[X]$ is a product of some of the $f_{\alpha}(X)$. A necessary and sufficient condition in order that there is $K\in[k,L]$ such that $g=f_K$ is gotten  for $k=\mathbb Q$ in  \cite[Remark 6]{HKN}, a result  without proof that we supply for an arbitrary field $k$.   

\begin{proposition}\label{4.24} If $g(X)\in L_u[X]$, there exists $K\in[k,L]$ such that $g=f_K$
 $\Leftrightarrow g\in \mathcal D\Leftrightarrow g(x)=0$ and $[L:K_g]=\deg (g)$.
\end{proposition}

\begin{proof}  If $g=f_K$ for some $K\in[k,L]$, then, $K=K_g$ by $(*)$. Obviously, $g(x)=0$.  Moreover, $[L:K]=\deg (f_K)=\deg (g)=[L:K_g]$.

Conversely, if  $g(x)=0$ and $[L:K_g]=\deg (g)$ hold, set $K:=K_g$. Then, $f_K(X)$ divides $g(X)$ in $K[X]$ since $g(x)=0$. Moreover, $[L:K]=\deg (f_K)=[L:K_g]=\deg (g)$, so that $g=f_K$. 
 
 If $g(X)=X-x$, we get that $L=K$. 
\end{proof} 

This result allows to characterize Boolean and finite separable  extensions using only polynomials with the following result.

\begin{theorem}\label{4.33} Let $k\subset L:=k[x]$ be a finite separable  field extension. The following conditions are equivalent:

\begin{enumerate}
\item $k\subset L$ is a Boolean extension;

\item For any $K\in[k,L[$, there is a unique subset $T$ of $\mathcal{CA}$ such that $f_ K=\sup\{m_{\beta}\mid E_{\beta}\in T\}$;

\item For any $g\in\mathcal D\setminus\{X-x\}$, there is a unique subset $I\subseteq\mathbb N_s$ such that $g=\sup\{m_{\beta}\mid   \beta\in I\}$. 
\end{enumerate} 
\end{theorem}

\begin{proof}    Theorem \ref{4.0203}, 
states that $k\subset L$ is  Boolean if and only if each $K\in[k,L[$ is of the form 
 $K=\cap_{E_{\beta}\in T}E_{\beta}$ for some unique $T\subseteq \mathcal{CA}$.
 
(1) $\Rightarrow$ (2) Assume first that $k\subset L$ is Boolean. 
Let $K\in[k,L[$ and $T:=\{E_{\gamma}\in\mathcal{CA}\mid K=\cap_{E_{\gamma}\in T}E_{\gamma}\}$, which is unique. Corollary \ref{4.251} yields that $f_K=\sup\{m_{\beta}\mid E_{\beta}\in T\}$. Let $T'\subseteq\mathcal {CA}$ be such that $f_K=\sup\{m_{\gamma}\mid E_{\gamma}\in T'\}$. Then, $K=\cap_{E_{\gamma}\in T'}E_{\gamma}$ and $T=T'$ follows  from the uniqueness property. 

(2) $\Rightarrow$ (1) Assume that for any $K\in[k,L[$, there is a unique subset $T$ of $\mathcal{CA}$ such that $f_K=\sup\{m_{\gamma}\mid E_{\gamma}\in T\}$. This implies that $K=\cap_{E_{\gamma}\in T}E_{\gamma}$ by Corollary \ref{4.251}.
Assume that there is some $T'\neq T$ such that $K=\cap_{E_{\gamma}\in T'}E_{\gamma} $ with $T'\subseteq \mathcal {CA}$. It follows that $f_K=\sup\{m_{\gamma}\mid E_{\gamma}\in T'\}$ and $T=T'$ because of the assumption on $T$. Therefore, $k\subset L$ is  Boolean by Theorem \ref{4.0203}. 
 
(2) $\Leftrightarrow$ (3) Use Corollary \ref{4.251} and the bijection $\varphi$. 
\end{proof}

 {\bf Scholium}. Here are the different steps in order to check that a finite separable extension $k\subset L$, with minimal polynomial $f(X)$, is Boolean according to Theorem \ref{4.33}:
\begin{enumerate}
\item Decompose   $f(X)$ into irreducible elements of $L[X]$.
\item Determine the set $\mathcal{E}$ of principal subfields.
\item Determine $\mathcal{CA}$ using Proposition \ref{4.26}.
\item Determine $\mathcal{D}$ using Proposition \ref{4.24}.
\item Check if condition (3) of Theorem \ref{4.33} holds.
\end{enumerate} 

 \begin{remark}\label{4.331} 
(1) Assume that $k\subset L$ is a finite separable Boolean field extension of degree $n$. Using the previous notation, we have $|\mathcal{CA}|=s$, which is also the value of $|\mathcal {A}|$. Using Theorem \ref{4.0} and Proposition \ref{4.21}, we get that $2^s\leq B_n$. It follows that if we want to calculate the elements of $\mathcal {CA}$, it is enough to calculate the $E_1,\dots,E_t$, and to stop as soon as we get $r$ distinct elements of $\mathcal {CA}$ such that $B_n<2^{r+1}$. 
 
(2) Let $k\subset L:=k[x]$ be a finite separable Boolean field extension and let $\mathcal{CA}=\{E_1,\dots,E_s\}$ be the set of co-atoms of the extension. Let $K:=k[z]\in]k,K[$. Then Example \ref{4.01} implies that, if $K=\cap [E_{\alpha}\in Y]$, where $Y\subseteq\mathcal{CA}$, we have $Y=\{E_{\alpha}\in\mathcal {CA}\mid z\in E_{\alpha}\}$. Moreover, $K^{\circ}=\cap [E_{\beta}\in\mathcal{CA}\setminus Y]$. But, since the extension is finite separable, there exists $y\in L$ such that $K^{\circ}=k[y]$. Then, $\mathcal{CA}\setminus Y=\{E_{\beta}\in \mathcal{CA}\mid y\in E_{\beta}\}=\{E_{\beta}\in\mathcal {CA}\mid z\not\in E_{\beta}\}$. It follows that $f_K=\sup(m_{\alpha}\mid E_{\alpha}\in Y)$ and $f_{k[y]}=\sup(m_{\beta}\mid E_{\beta}\in\mathcal{CA}\setminus Y)$. We recall that the $\sup$ is considered in $\mathcal D$, the set of the minimal polynomials of the elements of $[R,S]$.
\end{remark}

\begin{proposition}\label{4.34}  A finite separable  field extension $k\subset L:=k[x]$ such that $\mathcal D=\{g(X)\in L_u[X]\mid g(x)=0,\ g(X)|f(X)$ in $L[X]\}$ is  Boolean. If, in addition $k\subset L$ is Galois 
and  $k$ an infinite field, then $k\subset L$ is minimal of degree 2. 
\end{proposition}
\begin{proof}
For $\alpha\in\mathbb N_r$, we have $g_{\alpha}(X)=(X-x)f_{\alpha}(X)\in {\mathcal D}$, giving that there exists some $K\in[k,L]$ such that $g_{\alpha}=f_K$  with $L\neq K$ because $f_L(X)=X-x$. It follows that $K\subset L$ is minimal by \cite[Lemma 5.7]{Pic 6}. Hence, $K=E_{\gamma}$ for some $\gamma\in\mathbb N_s$.
Let $K'\in[k, L]$ and set $f_{K'}(X):=(X-x)\prod_{\alpha\in I}f_{\alpha}(X)$ for some $I\subseteq\mathbb N_r$. Moreover, $f_{K'}=$lcm$_{\alpha\in I}(\{m_{\alpha}\})$, for a unique $I$, and then a unique subset $T=\{L_{\alpha}\}_{\alpha\in I}$ of $\mathcal{CA}$ satisfying the hypotheses of Theorem \ref{4.33}.  (In fact, the $L_{\alpha}$ are 
 {\color{blue} all }
distinct and are the $E_{\beta}$.) Therefore, $k\subset L$ is Boolean and $\ell[k,L]=r$ by Theorem~\ref{4.0}, because $|\mathcal {CA}|=|\mathcal A|=r $.

Now if  $k\subset L$ is Galois, any $f_{\alpha}$ has degree 1. Set $n:=\deg(f)$, so that $r=n-1$, with the previous notation.
 Then $\ell[k,L]=n-1$ and $|[k,L]|=2^{n-1}$ by Theorem~\ref{4.0}. But $k$ is infinite, which implies that $2^{n-1}=|[k,L]|\leq 2^{n-2} +1$ by Proposition~\ref{4.21}, which gives $n=2$.
\end{proof}
There exist finite separable Boolean extensions $k\subset L$    such that $\mathcal D=\{g(X)\in L_u[X]\mid g(x)=0,\ g(X)|f(X)$ in $L[X]\}$ and  $k\subset L$ is not Galois. Take $k:=\mathbb{Q}$ and $L:=k[x]$, where $x:=\root 3\of 2$. Then $k\subset L$ is finite separable and not Galois, because not normal. Indeed, the minimal polynomial of $x$ is $f(X)=X^3-2=(X-x)(X^2+xX+x^2)$, with $X^2+xX+x^2$ irreducible in $L[X]$. Then, $k\subset L$ is Boolean by Proposition ~\ref{4.34}. 

Here is an example of Boolean extension where we show how the irreducible divisors of the minimal polynomial provides the subextensions of a finite separable extension of fields.

\begin{example}\label{4.35}  \cite[Remark 5.19]{Pic 4} Let $k:=\mathbb{Q}$,Ê$\ x:=\root 6\of 2$ and set $L:=k[x]$, which is a finite separable extension of $k$, 
 but not Galois. 

 The monic minimal polynomial of $x$ over $k$ is $f(x):=X^6-2=(X-x)(X+x)(X^2+x X+x^2)(X^2-x X+x^2)$, which is its decomposition into irreducible polynomials over $L$. Set $f_1(X):=X+x,\ f_2(X):=X^2+x X+x^2,\ f_3(X):=X^2-x X+x^2$ and $g_{\alpha}(X):=(X-x)f_{\alpha}(X)$, for $\alpha=1,2,3$. Then, $g_1(X)=X^2-x^2 =X^2-\root 3\of 2,\ g_2(X)=X^3-x^3=X^3-\sqrt 2$ and $g_3(X)=X^3-2x X^2 +2x^2X-x^3$. It follows that $K_1=k[\root 3\of 2],\ K_2=k[\sqrt 2]$ and $K_ 3=L$. Then, $L_1=K_1$ and $L_2=K_2$ by \cite[Lemma 5.10]{Pic 6}. Moreover, no subextension $K\in [k,L[ $ is such that $g_3=f_K$ since  $K_3=L$. Let $K\in[k,L]$ be such that $g_3$ divides strictly $f_K$ in $L[X] $. Then, $[L:K]=\deg (f_K)> 3$ gives that $[L:K]=6$, so that $K=k=L_3$ because $f_{L_3}=f$ \cite[Proposition 5.8]{Pic 6}. To end, $L_1\cap L_2=k=L_3$.
  Hence, $[k,L]=\{k,L_1,L_2,K\}$ is a Boolean lattice by Theorem ~\ref{4.03}, and we get the following diagram:

\centerline{$\begin{matrix}
   {}  &        {}      & L         &       {}       & {}     \\
   {}  & \nearrow & {}         & \nwarrow & {}     \\
L_1 &       {}       & {}         &      {}        & L_2 \\
  {}   & \nwarrow & {}        & \nearrow & {}     \\
  {}   &      {}        & k=L_3 & {}             & {} 
\end{matrix}$}
\end{example}
\begin{remark} \cite[Example 5.17 (2)]{Pic 6} Let $k:=\mathbb{Q}$ and $L=k[x]$, where $x:=\sqrt 2+\sqrt 3$. The monic minimal polynomial of $x$ over $k$ is $f(X)=X^4-10X^2+1=(X-x)(X+x)(X-x^{-1})(X+x^{-1})$. Set $f_1(X):=X+x,\ f_2(X)=X-x^{-1},\ f_3(X)=X+x^{-1}$. We get $K_1=L_1=k[\sqrt 6],\  K_2=L_2=k[\sqrt 3]$ and $K_3=L_3 =k[\sqrt 2]$. In particular, 
$L_{\alpha}=E_{\alpha}$ for each $\alpha$ and $L_{\alpha}\cap L_{\beta}=k$ for $\alpha\neq \beta,\ \alpha,\ \beta\in\mathbb N_3$, which shows that $[k,L]=\{k,L_1,L_2,L_3,L\}$ is not  Boolean
 because $\ell[k,L]=2$ and $|[k,L]|=5\neq 2^2$, but
 $k\subset L$ is  Galois. Therefore, although $k\subset L_{\alpha}$ is  Boolean for $\alpha\in\mathbb N_2$, the product $k\subset L_1L_2=L$ is not Boolean. We also observe that despite the fact that $k\subset L_{\alpha}$ and $L_{\alpha}\subset L$ 
 are Boolean 
for $\alpha\in\mathbb N_3,\ k\subset L$ is not  Boolean.
\end{remark}

However, we are able to  characterize   Boolean Galois extensions.

\begin{theorem}\label{4.37}  Let $k\subset L$ be  a finite  separable extension with normal closure  $N$. If $k\subset N$ is a cyclic extension with a square free degree, then $k\subset L$  is a Boolean extension.

In particular, a finite Galois extension $k\subset L$ with Galois group $G$ is  Boolean  if and only if $k\subset L$ is  cyclic  whose degree is square free.
\end{theorem}

\begin{proof}  We begin to prove the second part of the Theorem. Let $\mathcal G$ be the set of subgroups of $G$. For $H,H'\in \mathcal G$, we denote by $<H,H'>$ the subgroup of $\mathcal G$ generated by $H$ and $H '$. Define $\varphi:[k,L]\to\mathcal G$ by $\varphi (K):=$Aut$_K(L)$, the group of $K$-automorphisms of $L$, for each $K\in[k,L]$ and $\psi:\mathcal G\to[k,L]$  by $\psi(H):=$Fix$(H)$, the fixed field of $H$ in $L$, for each $H\in\mathcal G$. The Fundamental Theorem of Galois Theory for finite extensions shows that $\varphi$ and $\psi$ are reversing order isomorphisms of lattices, with $\varphi=\psi^{-1}$. 
 Therefore,  $[k,L]$ is a Boolean lattice if and only if $\mathcal G$ is a Boolean lattice. 
To conclude, use 
 \cite[Corollary 2]{W} which says that for    a finite group $G$,  the lattice of its subgroups is a Boolean lattice if and only if $G$ is a cyclic group whose order is square free.

Now, if $k\subset L$ is  a finite  separable extension whose normal closure is $N$ such that $k\subset N$ is a cyclic extension with a square free degree, then $k\subset L$  is Boolean by Proposition \ref{4.021} because $k\subset N$ is  Boolean.
\end{proof}

\begin{remark} \label{4.372} The last  part  of the Theorem has no converse, unless adding some new assumptions as in Theorem \ref{4.03}. Consider Example \ref{4.35}. The normal closure $N$ of the extension $k\subset L$ is generated over $k$ by $x$ and its conjugates, which are the zeroes $\{\pm x,\pm jx,\pm j^2x\}$, of $f(X)=X^6-2$, where $j=(1+i\sqrt 3)/2$. Then, $N=k[x,jx,j^2x]$. Moreover, $k\subset N$ is Galois. Assume that $k\subset N$ is Boolean. Then, $k\subset N$ is a cyclic extension by Theorem \ref{4.37}. Set $K_1':=k[j^2x^2]\subset N$. Since $(j^2x^2)^3=2$, we get that $[K'_1:k]=3=[K_1:k]$, and we have two subextensions of $k\subset N$ of degree 3, a contradiction for a cyclic extension (see \cite[AV, page 81]{Bki A}). Then, $k\subset N$ is not Boolean. 
\end{remark}

The two next examples exhibit Galois Boolean field extensions.

\begin{example}\label{4.38} (1) Let $n$ be a positive integer, $n\geq 2$. In view of \cite[AV.152, Exercice 3)]{Bki A}, there exists a cyclic extension of $\mathbb{Q}$ of degree $n$. It is enough to take a square free integer $n$ and to use Theorem \ref{4.37} to get a Boolean extension.

(2) Let $k:=\mathbb{F}_2=\mathbb{Z}/2\mathbb{Z}$ be the finite field with two elements, and let $K_n$ be the 
 cyclic 
extension of $k$ of degree $n$. Set $L:=K_{30}$. The subfields of $L$ are the $K_n$, where $n$ divides 30. In view of Theorem \ref{4.37}, $k\subset L$ is a Boolean extension, because cyclic of degree a square free integer, and $[k,L]=\{K_n\mid  n=1,2,3,5,6,10,15,30\}$. 
\end{example}

 \begin{proposition}\label{4.373} Let $k\subset L$ be a finite   Galois Boolean extension and let $T,U\in]R,S[$. Then $U=T^{\circ}$ if and only if $k\subset T$and $k\subset U$ are linearly disjoint with $L=TU$.
\end{proposition}

\begin{proof} If $k\subset T$ and $k\subset U$ are linearly disjoint, then $T\cap U=k$, so that $U=T^{\circ}$ since $TU=L$.

Conversely, assume that $U=T^{\circ}$. We are going to show how $U$ is build from $T$. Since $k\subset L$ is a finite Galois Boolean extension, Theorem \ref{4.37} shows that $k\subset L$ is cyclic, $n:=[L:k]$ is square free, and so is $[T:k]$. Set $n=p_1\cdots p_kp_{k+1}\cdots p_r$ where the $p_i$'s are distinct prime integers ordered such that $m:=[T:k]=p_1\cdots p_k$, and set $l:=n/m$. Since $k\subset L$ is cyclic, there exists $V\in[k,L]$ such that $[V:k]=l$. It follows that $(m,l)=([T:k],[V:k])=1$, so that $V\cap T=k$ and $TV:=L$. Indeed, $T,V\subset TV\subseteq L$ shows  that $m,l$ dividing $[TV:k]$, gives that $[L:k]=n=ml$ divides $[TV:k]$, which leads to $TV=L$. Then, $V=T^{\circ}=U$. Under these conditions, we have $[TU:k]=[L:k]=ml=[T:k][U:k]$, which shows that $k\subset T$ and $k\subset U$ are linearly disjoint.
\end{proof}

 We can say more about distributive Galois extension non necessarily Boolean, involving a   result from Dobbs-Mullins. 

\begin{proposition}\label{cyclic} Let $k\subset L$ be a finite  Galois extension with degree  $n ( =\prod_{i=1}^mp_i^{e_i}$ the factorization into prime integers).
 
(1) \cite[Proposition 2.2]{DM} If $k\subset L$ is Abelian, then $\ell[k,L]=\sum_{i=1}^me_i$.

 (2)  If in addition  $k\subset L$ is distributive, then $k\subset L$ is cyclic, $\ell[k,L]=\sum_{i=1}^me_i$ and $|[k,L]|=\tau(n)$, where $\tau(n)$ is the number of divisors of $n$.
 
 (3)  If in addition $k\subset L$ is Boolean,  $\ell[k,L]=m$ and $|[k,L]|=2^m$. 
  
\end{proposition}

\begin{proof} (2) Let $G$ be the Galois group of $k\subset L$ and $\mathcal G$ be the set of subgroups of $G$. Since $k\subset L$ is  distributive, so is $\mathcal G$. Then, $G$ is cyclic by \cite[Page 97]{R}, and so is the extension $k\subset L$. The first part of (2) comes from (1) since a cyclic extension is Abelian. Moreover, there is a bijection between the subgroups of a cyclic group of order $n$ and the divisors of $n$, as there is a bijection between  $\mathcal G$ and $[k,L]$. This gives the last equality. 

(3) Assume that $k\subset L$ is Boolean,  then cyclic by Theorem  \ref{4.37} and $n$ is square free, so that $e_i=1$ for each $i$. Hence, $\ell[k,L]=m$ and $|[k,L]|=2^m$, since $k\subset L$ is Boolean by Theorem \ref{4.0}. 
\end{proof} 

In a recent paper \cite{Pic 6}, we characterized ring extensions $R\subset S$ of length 2 and gave the value of $|[R,S]|$. It is then easy to characterize an extension of length 2 which is Boolean. 

\begin{proposition}\label{4.340} Let $R\subset S$ be an FIP   extension of length 2. Then $R\subset S$ is 
Boolean if and only if $|[R,S]|=4$,  and,  if and only if one of the following condition holds:
 \begin{enumerate}
\item  $|\mathrm{Supp}(S/R)|=2$ and $ \mathrm{Supp}(S/R)\subseteq \mathrm{Max}(R)$. 

\item $R\subset S$ is infra-integral such that $\mathrm{Supp}(S/R)=\{M\},\ {}_S^+R\neq R,S$ and $ (R:S)= M$. 

\item  $R\subset S$ is  t-closed integral such that $\mathrm{Supp}(S/R)=\{M\},\ M=(R:S)\in \mathrm{Max}(S)$, and one of the following conditions holds:
\end{enumerate}

\noindent (a) $R/M\subset S/M$ is neither radicial nor separable, nor exceptional.

\noindent (b) $R/M\subset S/M$ is a finite separable field extension and $t=2$, where $t$ is the number of principal subfields of $S/M$ different from $R/M$.

\end{proposition} 

\begin{proof} Assume that $R\subset S$ is Boolean. Then $|[R,S]|=4$ by Theorem \ref{4.0}. Conversely, if $|[R,S]|=4$, then $[R,S]=\{R,T,U,S\}$ for some $U,T\in]R,S[$, where $T$ and $U$ are incomparable, so that $R\subset S$ is Boolean by Theorem \ref{4.03}. Now the second equivalence comes from \cite[Theorem 6.1]{Pic 6} which gives the different cases for a length 2 extension $R\subset S$ to satisfy $|[R,S]|=4$.
\end{proof}

\end{document}